\documentclass{amsart}
\usepackage{amsmath,amssymb,amsthm}
\usepackage{graphicx}
\usepackage{bm}
\usepackage{braket}

\newtheorem{theorem}{Theorem}[section]
\newtheorem{proposition}[theorem]{Proposition}

\newtheorem{lemma}[theorem]{Lemma}
\theoremstyle{definition}
\newtheorem{definition}[theorem]{Definition}
\newtheorem{remark}[theorem]{Remark}
\newtheorem{example}[theorem]{Example}

\numberwithin{equation}{section}%

\newcommand{\mgr}{MGR}
\newcommand{\oss}{oriented spatial surface}
\newcommand{\diag}{Y-oriented diagram}
\newcommand{\surf}{oriented spatial surface with an $S^1$-orientation}
\newcommand{\surfs}{oriented spatial surfaces with $S^1$-orientations}
\newcommand{\Reidemeister}{Y-oriented Reidemeister moves on $S^2$}
\newcommand{\ccycl}[2]{\sum_{{p}=1}^{#1}\sum_{{q}=1}^{#2}\theta\left(\la\upsilon\star^{{p}-1}a\star^{{q}-1}{b}\ra\la{a}*^{{q}-1}{b}\ra\la{b}\ra\right)}%
\newcommand{\oline}[1]{\mathbin{\overline{#1}}}
\newcommand{\otilde}[1]{\mathbin{\widetilde{#1}}}
\newcommand{\PT}{\widetilde{\partial}}
\newcommand{\la}{\langle}
\newcommand{\ul}{\underline}
\newcommand{\ra}{\rangle}
\newcommand{\lr}[1]{\langle#1\rangle}
\newcommand{\col}{\operatorname{Col}}
\newcommand{\type}{\operatorname{type}}
\newcommand{\Z}{\mathbb{Z}}
\newcommand{\R}{\mathbb{R}}
\newcommand{\si}{c}
\newcommand{\Wai}{Y}
\newcommand{\ai}{i}
\newcommand{\jei}{j}
\newcommand{\kei}{k}

\makeatletter
\@namedef{subjclassname@2020}{\textup{2020} Mathematics Subject Classification}
\makeatother
\begin{document}

\title[Cohomology of racks and multiple group racks]{(Co)homology of racks and multiple group racks for compact oriented surfaces in the 3-sphere}
\author[S.~Matsuzaki]{Shosaku Matsuzaki}
\address[S.~Matsuzaki]{Liberal Arts Education Center, Ashikaga University, 268-1 Ohmae-cho, Ashikaga-shi, Tochigi 326-8558, Japan}
\email{matsuzaki.shosaku@g.ashikaga.ac.jp}
\author[T.~Murao]{Tomo Murao}
\address[T.~Murao]{Science and Technology Unit, Natural Sciences Cluster, Research and Education Faculty, Kochi University, 2-5-1 Akebono-cho, Kochi-shi, Kochi 780-8072, Japan}
\email{tmurao@kochi-u.ac.jp}
\keywords{rack; multiple group rack; \surf; cocycle invariant}
\subjclass[2020]{Primary~57K12, Secondary~57K10}
\thanks{This work was supported by JSPS KAKENHI (Grant Numbers: JP20K22312, JP21K13796).}
\maketitle

\begin{abstract}
We introduce (co)homology theory for multiple group racks and construct cocycle invariants of compact oriented surfaces in the 3-sphere using their 2-cocycles, where a multiple group rack is a rack consisting of a disjoint union of groups. We provide a method to find new rack cocycles and multiple group rack cocycles from existing rack cocycles and quandle cocycles. Evaluating cocycle invariants, we determine several symmetry types, caused by chirality and invertibility, of compact oriented surfaces in the 3-sphere.
\end{abstract}

\section{Introduction}\label{SEC:intro}
In the study of low dimensional topology, it is important to understand the isotopy classes of (incompressible) surfaces embedded in a 3-manifold.
In particular, in knot theory, we can obtain rich information about a knot from its Seifert surface, which is an orientable surface whose boundary is the knot.
There are many studies about the isotopy classes of Seifert surfaces of knots:
any fibered knot has only a single isotopy type of minimal genus Seifert surfaces~\cite{BZ1967}, and A.~Hatcher and W.~Thurston~\cite{HT1985} classified the isotopy classes of minimal genus Seifert surfaces for 2-bridge knots, etc.
The abundance of Seifert surfaces later led O.~Kakimizu introduce a simplicial complex that records the structure of Seifert surfaces, whose vertices are isotopy classes of minimal genus Seifert surfaces of a given knot~\cite{K1992}.
The complex, called the Kakimizu complex, has been determined for prime knots of 10 or less crossings~\cite{K2005} by using the theory of sutured manifolds established by D.~Gabai~\cite{G1984}.
\par
In this paper, we study the isotopy classes of oriented compact surfaces embedded in $S^3$, called oriented spatial surfaces, which do not necessarily have the boundaries.
A closed spatial surface can be regarded as a generalization of a handlebody-knot by taking its boundary, where a handlebody-knot is a handlebody embedded in $S^3$, which is a generalization of a knot to higher genera.
It is known that if the complements of two knots are homeomorphic, then the knots are ambient isotopic up to their mirror images~\cite{GL1989}.
However, M.~Motto~\cite{M1990} gave infinitely many different handlebody-knots whose complements are homeomorphic by using a cut and paste argument and a study of mapping class groups of surfaces.
Further, we can quite easily have different spatial surfaces whose complements are homeomorphic.
For example, an annulus can be embedded in $S^3$ in infinite ways by only twisting it so that its complements are homeomorphic.
Consequently, we need powerful tools having more abundant information than the topology of the complement for classifications of spatial surfaces.
\par
D.~Joice~\cite{Joyce82} and S.~V.~Matveev~\cite{Matveev82} independently introduced a quandle, 
which is an algebra whose axioms correspond to the Reidemeister moves for knots.
A quandle yields various invariants of knots such as coloring numbers, (shadow) quandle  cocycle invariants, etc.
Similarly, other algebras: a rack~\cite{FennRourke92}, birack~\cite{FRS1993}, biquandle~\cite{FJK2004,FennRourkeSanderson95,KR2003}, multiple conjugation quandle~\cite{Ishii15-1}
and multiple conjugation biquandle~\cite{IIKKMO2018_1,IN2017} were introduced, whose axioms correspond to the Reidemeister moves for related topological objects:
knots, surface-knots, framed knots and handlebody-knots.
Recently, the first author~\cite{Matsuzaki19} introduced the Reidemeister moves of \oss s, and the authors and A.~Ishii~\cite{IMM20} extracted an algebra, called a multiple group rack (\mgr), from them.
Using \mgr s, they constructed an invariant called a coloring number and distinguished some \oss s.
\par
R.~Fenn et al.~\cite{FennRourkeSanderson95} introduced a chain complex for a rack, and J.~S.~Carter et al.~\cite{CJKLS2003} developed the (co)homology theory for a quandle.
The motivation of the theory was to construct invariants for knots and surface-knots.
According to the theory, we can obtain a powerful invariant, called a (shadow) quandle cocycle invariant, if a 2-cocycle of the chain complex is given.
Other (co)homology theory are also developed on biracks and biquandles~\cite{CES2004,CEGN2014}, multiple conjugation quandles~\cite{CIST2017} and multiple conjugation biquandles~\cite{IIKKMO2018,IIKKMO2021} 
to provide invariants of the corresponding topological objects (cf.~\cite{AS2005,CJKLS2003,IshiiIwakiriJangOshiro13}).
In this paper, we introduce (co)homology theory for \mgr s and construct (shadow) \mgr{} cocycle invariants of \oss s.
\par
Although it is not easy to find rack 2-cocycles or \mgr{} 2-cocycles in general, we introduce a way to construct a ``new'' rack 2-cocycle from an arbitrary rack 2-cocycle.
Further, we also introduce a way to construct an \mgr{} 2-cocycle from a rack 2-cocycle.
Calculating MGR cocycle invariants, we distinguish some \oss s which can never be distinguished by MGR coloring numbers.
Moreover, we consider several symmetries that are naturally caused by both chirality and invertibility of \oss s.
We give examples of \oss s with the symmetries, respectively.
\par
This paper is organized as follows.
In Section~\ref{SEC:def_of_sp_surf}, we recall the notion of \oss s and introduce $S^1$-orientations of them.
In Section~\ref{SEC_def_mgr}, we recall racks and \mgr s and introduce colorings for a diagram of \surfs.
In Section~\ref{sec:homology}, we introduce a (co)homology theory of \mgr s and define shadow \mgr{}  cocycle invariants.
In Section~\ref{SEC:2-COCYCLE_racks}, we recall a rack structure of $Q^n$ for a rack $Q$ 
and construct rack 2-cocycles of $Q^n$.
In Section~\ref{SEC:2-COCYCLE_mgr}, we provide a method to construct \mgr{} 2-cocycles from a given rack 2-cocycles.
In Section~\ref{SEC:SYMMETRIES}, we study five symmetries for \oss s and give a table of arbitrary genus \oss s with the symmetries by using our shadow \mgr{} cocycle invariants.

\section{An \surf}\label{SEC:def_of_sp_surf}
An \textit{\oss} is a compact oriented surface embedded in the 3-sphere $S^3=\R^3\cup\{\infty\}$.
Two \oss{s} $F$ and $F'$ are \textit{equivalent}, denoted by $F\cong{F}'$,
if there is an orientation-preserving self-homeomorphism $h$ of $S^3$ such that $h(F)=F'$,
where we note that the orientation of $h(F)$ also coincides with that of $F'$.
Unless otherwise noted, we suppose the following conditions.
\begin{itemize}
\item Each component of any \oss{} $F$ has non-empty boundary.
\item $F$ has no disk components.
\end{itemize}
The \oss{} obtained by reversing the orientation of $F$ is the \textit{inverse}\footnote{The inverse of $F$ has called  the reverse of $F$ in \cite{IMM20}.} of $F$,
denoted by $-F$.
For an orientation-reversing self-homeomorphism $r$ of $S^3$,
the \oss{} $r(F)$ is the \textit{mirror image} of $F$,
denoted by $F^*$.
In figures of this paper,
the front side and the back side of an \oss{} are colored by light gray and dark gray,
respectively (see Fig.~\ref{FIG:color_gray}).

		\begin{figure}[htbp]\centerline{
		\includegraphics{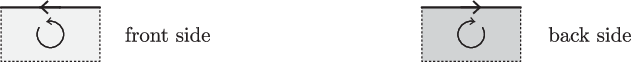}}
		\caption{The front side and the back side.}
		\label{FIG:color_gray}
		\end{figure}

A \textit{spatial trivalent graph} is a trivalent finite graph, whose vertices are of valency 3, embedded in $S^3$.
A trivalent graph may have a circle component, which has no vertices.
Let $D$ be a diagram of a spatial trivalent graph.
We obtain an \oss{} $F$ from $D$ by taking a regular neighborhood of $D$ in $S^2=\R^2\cup\{\infty\}$ and perturbing it around all crossings of $D$ according to its over/under information, and we give an orientation so that the front side of $F$ faces into the positive direction of the $z$-axis of $\R^3$ (see Fig.~\ref{FIG:construction_of_spatial_surface}).
Then we say that $D$ \textit{represents} the \oss{} $F$ and call $D$ a \textit{diagram} of $F$.
Any \oss{} is equivalent to an \oss{} obtained by the process (see~\cite{Matsuzaki19}).

		\begin{figure}[htbp]\centerline{
		\includegraphics{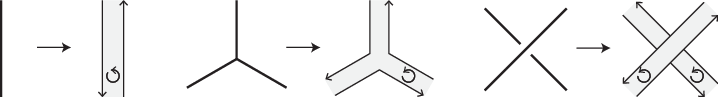}}
		\caption{The process for obtaining an \oss.}
		\label{FIG:construction_of_spatial_surface}
		\end{figure}

\begin{theorem}[\cite{Matsuzaki19}]\label{theorem_oriented_surface}
Two \oss{s} are equivalent if and only if their diagrams are related by finitely many Reidemeister moves on $S^2$ depicted in Fig.~\ref{FIG:spatial_surface_Reidemeister_move}.
\end{theorem}

		\begin{figure}[htbp]\centerline{
		\includegraphics{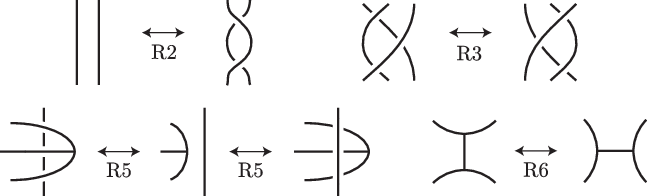}}
		\caption{The Reidemeister moves for \oss{s}.}
		\label{FIG:spatial_surface_Reidemeister_move}
		\end{figure}

An \textit{$S^1$-orientation} of an \oss{} $F$ is a collection of orientations of the cores of all annulus components of $F$.
Two \surfs{} $(F,o)$ and $(F',o')$ are \textit{equivalent},
denoted by $(F,o)\cong(F',o')$,
if there is an orientation-preserving self-homeomorphism $h$ of $S^3$
such that $h(F)=F'$ and $h$ sends $o$ to $o'$.
We regard an $S^1$-orientation $o$ of $F$ as an oriented link contained in $F$.
Then the inverse, denoted by $-o$,
is also an $S^1$-orientation of $F$,
and the mirror image, denoted by $o^*$, is an $S^1$-orientation of $F^*$.
For example, we have \surfs{} $(F,-o)$, $(-F,o)$, $(-F,-o)$, $(F^*,o^*)$, $(F^*,-o^*)$, $(-F^*,o^*)$ and $(-F^*,-o^*)$ from $(F,o)$ as illustrated in Fig.~\ref{fig:s1orientation}.

	\begin{figure}[htpb]\centerline{
	\includegraphics{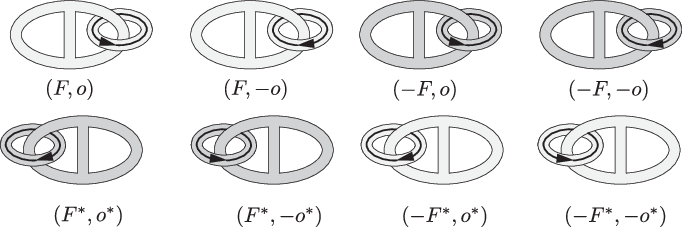}}
	\caption{Examples of \surf.}
	\label{fig:s1orientation}
	\end{figure}
	
\begin{remark}\label{REM:S1orientation}{\rm
We can regard a framed oriented link as an \surf,
since we can reverse the orientation of any annulus component $C$ by an isotopy in the regular neighborhood of $C$ while keeping its core orientation.
}\end{remark}

Let $D$ be a diagram of an \oss{} $F$.
A \textit{Y-orientation} of $D$ is a collection of orientations of all edges and circle components of $D$ without sources and sinks with respect to the orientation as shown in Fig.~\ref{fig:Y_orientations}.
Suppose that $D$ is equipped with a $Y$-orientation;
we obtain the $S^1$-orientation $o$ of $F$ that agrees with the Y-orientation of circle components of $D$.
Then we say that $D$ \textit{represents} the \surf{} $(F,o)$ and call $D$ a \textit{\diag} of $(F,o)$.
Any \surf{} is represented by a \diag.
\textit{Y-oriented Reidemeister moves} are the moves depicted in Fig.~\ref{FIG:spatial_surface_Reidemeister_move} between two \diag s that are identical except in the disk where the move is applied.

	\begin{figure}[htpb]\centerline{
	\includegraphics{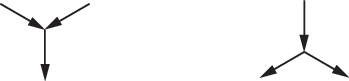}}
	\caption{Y-orientations.}
	\label{fig:Y_orientations}
	\end{figure}

\begin{proposition}[\cite{Ishii15-2}]\label{proposition_Y_orientation_relation}
Any two Y-orientations of a diagram of \oss s can be transformed into each other by finitely many \Reidemeister{}  and reversing the orientations of some circle components.
\end{proposition}

The following theorem immediately follows from Theorem~\ref{theorem_oriented_surface} and Proposition~\ref{proposition_Y_orientation_relation}.

\begin{theorem}\label{theorem_oriented_surface_with_S1}
Two \surfs{} are equivalent if and only if their \diag{s} are related by finitely many \Reidemeister{}.
\end{theorem}

\section{An \mgr{} coloring for \surfs{}}\label{SEC_def_mgr}
In this section, we recall a notion of multiple group racks.
A \textit{rack}~\cite{FennRourke92} is a non-empty set ${Q}$ with a binary operation $*:{Q}^2\to{Q}$ satisfying the following axioms.
\begin{itemize}
\item For any ${a}\in{Q}$,
the map $S_{{a}}:{Q}\to{Q}$ defined by $S_{{a}}({x})={x}*{a}$ is bijective.
\item For any ${a},{b},{c}\in{Q}$,
$({a}*{b})*{c}=({a}*{c})*({b}*{c})$.
\end{itemize}
A rack ${Q}$ is a \textit{quandle}~\cite{Joyce82,Matveev82} if ${a}*{a}={a}$ for any ${a}\in{Q}$.
For a rack ${Q}$, a \textit{$Q$-set} is a non-empty set $\Wai$ with a map $\star:\Wai\times{Q}\to\Wai;(\upsilon,a)\mapsto\upsilon\star a$ satisfying the following axioms.
\begin{itemize}
\item For any ${a}\in{Q}$,
the map $T_{{a}}:\Wai\to\Wai$ defined by $T_{{a}}(\upsilon)=\upsilon\star{a}$ is bijective.
\item For any $\upsilon\in\Wai$ and ${a},{b}\in{Q}$,
$(\upsilon\star{a})\star{b}=(\upsilon\star{b})\star({a}*{b})$.
\end {itemize}
A rack ${Q}$ itself is a ${Q}$-set with its binary operation:
$\upsilon\star{a}=\upsilon*{a}$ for any $\upsilon,{a}\in{Q}$.
Any singleton set $\{\upsilon\}$ is a ${Q}$-set with the map $\star$ defined by $\upsilon\star{a}=\upsilon$ for ${a}\in{Q}$, which we call a \textit{trivial} ${Q}$-set.
We denote $S_{{a}}^{{i}}({b})$ by $b*^{{i}}{a}$ for any ${a},{b}\in{Q}$ and ${i}\in\Z$, where we note that $S_{{a}}^0=\mathrm{id}_{{Q}}$.
We denote $T_{{a}}^{{i}}(\upsilon)$ by $\upsilon\star^{{i}}{a}$ for any $\upsilon\in\Wai$, ${a}\in{Q}$ and ${i}\in\Z$, where we note that ${T}_{{a}}^0=\mathrm{id}_{\Wai}$.
For any $\upsilon\in\Wai$, ${a},{a}_1,\ldots,{a}_n\in{Q}$ and ${i}_1,\ldots,{i}_n\in\Z$, we often abbreviate parentheses and write
${S}_{{a}_n}^{{i}_n}\circ\cdots\circ{S}_{{a}_1}^{{i}_1} ({a})$ by
${a}*^{{i}_1}{a}_1*^{{i}_2}\cdots*^{{i}_n}{a}_n$, and we write
${T}_{{a}_n}^{{i}_n}\circ\cdots\circ{T}_{{a}_1}^{{i}_1} ({\upsilon})$ by
${\upsilon}\star^{{i}_1}{a}_1\star^{{i}_2}\cdots\star^{{i}_n}{a}_n$.
\par
We define the \textit{type of a rack ${Q}$ with a ${Q}$-set $\Wai$} by
\[
\type{Q_{\Wai}}=\min\Set{n>0|
\begin{array}{l}
\text{${a}*^n{b}=a$ for any ${a},{b}\in{Q}$,}\\
\text{$\upsilon\star^n{c}=\upsilon$ for any $\upsilon\in\Wai$ and ${c}\in{Q}$}
\end{array}
},
\]
where we set $\min\emptyset:=\infty$.
We write $\type{{Q}_{\{\upsilon\}}}$ by $\type{{Q}}$ and call it the \textit{type of a rack ${Q}$},
where we note that $\type{Q}$ divides $\type{Q}_\Wai$.
For $q\in Q$, we define $\type q:=\min\Set{n>0|q*^nq=q}$.
If ${Q}$ and $\Wai$ are finite,
$\type{Q_{\Wai}}$ is finite since the set $\{S_{{a}}^n\mid{n}\in\Z\}\cup\{T_{{a}}^n\mid{n}\in\Z\}$ is finite for any ${a}\in{Q}$.
\par
We give some examples of racks and quandles.
We define the binary operation $*:\Z_n^2\to\Z_n$ by ${i}*{j}:=2{j}-{i}$ for any ${i},{j}\in\Z_n$,
where $\Z_n$ is the cyclic group $\Z/n\Z$ of order ${n}$.
Then $\Z_n$ is a quandle, called the \textit{dihedral quandle} of order ${n}$ and denoted by $R_{{n}}$.
Let ${R}$ be a ring and ${M}$ a left $R[t^{\pm1},s]/(s(t+s-1))$-module.
We define a binary operation $*:{M}^2\to{M}$ by $x*y:=tx+sy$ for any $x,y \in{M}$.
Then $M$ is a rack, called the \textit{$(t,s)$-rack}.
When $s=1-t$, the $(t,s)$-rack is called the \textit{Alexander quandle}.

\begin{definition}[\cite{IMM20}]\label{def:multiplegrouprack}
A \textit{multiple group rack (\mgr)} ${X}=\bigsqcup_{\lambda\in\Lambda}{G}_{\lambda}$ is a disjoint union of groups ${G}_{\lambda}$ ($\lambda\in\Lambda$) with a binary operation $*:{X}^2\to{X}$ satisfying the following axioms.
\begin{enumerate}
\item For any ${x}\in{X}$ and ${y}_1,{y}_2\in{G}_{\lambda}$,
${x}*({y}_1{y}_2)=({x}*{y}_1)*{y}_2$ and ${x}*e_{\lambda}={x}$,
where $e_{\lambda}$ is the identity of ${G}_{\lambda}$.
\item For any ${x},{y},{z}\in{X}$, $(x*y)*z=(x*z)*(y*z)$.
\item For any ${x}_1,{x}_2\in{G}_{\lambda}$ and ${y}\in{X}$,
$({x}_1{x}_2)*{y}=({x}_1*{y})({x}_2*{y})$,
where ${x}_1*{y}$, ${x}_2*{y}\in{G}_\mu$ for some $\mu\in\Lambda$.
\end{enumerate}
\end{definition}

An \mgr{} is a rack with the binary operation $*$.
An \mgr{} $X=\bigsqcup_{\lambda\in\Lambda}G_\lambda$ is called a \textit{multiple conjugation quandle}~\cite{Ishii15-1} if ${x}_1*{x}_2={x}_2^{-1}{x}_1{x}_2$ for any ${x_1},{x_2}\in{G}_\lambda$.
\par
An \textit{$X$-set of an \mgr{}} ${X}=\bigsqcup_{\lambda \in \Lambda}{G}_{\lambda}$ is a non-empty set $\Wai$ with a map $\star:\Wai\times{X}\to\Wai;(\upsilon,x)\mapsto \upsilon\star x$ satisfying the following axioms.
\begin{enumerate}
\item For any $\upsilon\in\Wai$ and ${x}_1,{x}_2\in{G}_{\lambda}$,
$\upsilon\star e_{\lambda}=\upsilon,\quad\upsilon\star({x}_1{x}_2)=(\upsilon\star{x}_1)\star{x}_2$,
where $e_\lambda$ is the identity of $G_\lambda$.
\item For any $\upsilon\in\Wai$ and ${x},{y}\in{X}$,
$(\upsilon\star{x})\star{y}=(\upsilon\star{y})\star({x}*{y})$.
\end{enumerate}
By the first axiom, the map $T_x:\Wai\to\Wai;\upsilon\to\upsilon\star{x}$ is bijective for any $x\in{X}$.
An \mgr{} $X=\bigsqcup_{\lambda\in\Lambda}G_{\lambda}$ itself is an $X$-set with its binary operation:
$\upsilon\star{x}=\upsilon*{x}$ for any $\upsilon,{x}\in{X}$.
Any singleton set $\{\upsilon\}$ is an ${X}$-set with the map $\star$ defined by $\upsilon\star{x}=\upsilon$ for ${x}\in{X}$, which is called a \textit{trivial} ${X}$-set.
The index set $\Lambda$ is an ${X}$-set with the map $\star$ defined by $\lambda\star{x}=\mu$ when $S_{{x}}({G}_{\lambda})\subset{G}_{\mu}$ for $\lambda$, $\mu\in\Lambda$ and ${x}\in{X}$.
\par
Let ${G}$ be a group with the identity $e$.
A \textit{${G}$-family of racks} \cite{IMM20} $({Q},\{*^{g}\}_{{g}\in{G}})$ is a set ${Q}$ with a family of binary operations $*^g:{Q}^2\to{Q}$ ($g\in{G}$) satisfying the following axioms.
\begin{itemize}
\item For any ${a},{b}\in{Q}$ and ${g},{h}\in{G}$,
${a}*^{{g}{h}}{b}=({a}*^g{b})*^h{b}\text{ and } {a}*^e{b}={a}$.
\item For any ${a},{b},{c}\in{Q}$ and ${g},{h}\in{G}$,
$({a}*^{g}{b})*^{h}{c}=({a}*^{h}{c})*^{{h}^{-1}{g}{h}}({b}*^{h}{c})$.
\end{itemize}
A $G$-family of racks $({Q},\{*^{g}\}_{{g}\in{G}})$ is called a \textit{$G$-family of quandles}~\cite{IshiiIwakiriJangOshiro13} if ${a}*^{g}{a}={a}$ for any ${a}\in{Q}$ and ${g}\in{G}$.
A ${G}$-family of racks $({Q},\{*^{{g}}\}_{{g}\in{G}})$ yields an \mgr{} $Q\times G=\bigsqcup_{{a}\in{Q}}\left(\{a\}\times{G}\right)$ with
\[
({a},{g})*({b},{h})=\left({a}*^{h}{b},{h}^{-1}{g}{h}\right),\quad
({a},{g})({a},{h})=({a},{g}{h})
\]
for any ${a},{b}\in{Q}$ and ${g},{h}\in{G}$ (see~\cite{IMM20}).
For a rack $Q$ and any $k\in\Z_{>0}\cup\{\infty\}$,
$(Q,\{*^{i}\}_{i\in\Z_{k\type{Q}}})$ is a $\Z_{k\type{Q}}$-family of racks, where we set $\Z_\infty=\Z$.
Then, we obtain the \mgr{} $Q\times\Z_{k\type{Q}}=\bigsqcup_{{a}\in{Q}}\left(\{{a}\}\times\Z_{k\type{Q}}\right)$, called the \textit{associated \mgr{}} of $Q$.

\begin{remark}\label{frequentlyused_mgr}
Let $\Wai$ be a $Q$-set and $Q\times\Z_{k\type{Q_\Wai}}$ the associated \mgr.
Then $\Wai$ is also a $(Q\times\Z_{k\type{Q_\Wai}})$-set with the map $\star:\Wai\times (Q\times\Z_{k\type{Q_\Wai}})\to \Wai$ defined by ${\upsilon}\star\left(a,i\right)={\upsilon}\star^i{a}$ for any $\upsilon\in\Wai$, $a\in Q$ and $i\in\Z$.
In this paper, we often use the \mgr{} $Q\times\Z_{k\type{Q_\Wai}}$ and the $Q\times\Z_{k\type{Q_\Wai}}$-set (see Proposition~\ref{prop_main_mgr_ccycle}).
\end{remark}

Let $D$ be a \diag{} of an \surf{} $(F,o)$.
We denote by $\mathcal{A}(D)$ the set of arcs of $D$, where an arc is a piece of a curve each of whose endpoints is an undercrossing or a vertex.
We denote by $\mathcal R(D)$ the set of complementary regions of $D$.
In this paper, a direction of an arc is often represented by the normal orientation, which is obtained by rotating the usual orientation counterclockwise by $\pi/2$ on the diagram.
\par
Let $X=\bigsqcup_{\lambda\in\Lambda}G_\lambda$ be an \mgr.
An \textit{$X$-coloring} of $D$ is a map $\si:\mathcal{A}(D)\to{X}$ satisfying the conditions (i), (ii) and (iii) depicted in Fig.~\ref{FIG:X_Y_coloring} at each crossing and vertex of $D$.
We denote by $\col_X(D)$ the set of $X$-colorings of $D$.
The cardinality $\#\col_X(D)$ is an invariant of $(F,o)$ \cite{IMM20}.
Let $\Wai$ be an $X$-set.
An \textit{$X_{\Wai}$-coloring} of $D$ is a map $\si:\mathcal A(D)\cup\mathcal R(D)\to X\cup \Wai$ such that the restriction $\si|_{\mathcal A(D)}$ is an $X$-coloring of $D$ and that the condition (iv) in Fig.~\ref{FIG:X_Y_coloring} is satisfied at any adjacent regions and the arc between them.
We denote by $\col_X(D)_{\Wai}$ the set of $X_{\Wai}$-colorings of $D$.

	\begin{figure}[htpb]\centerline{
	\includegraphics{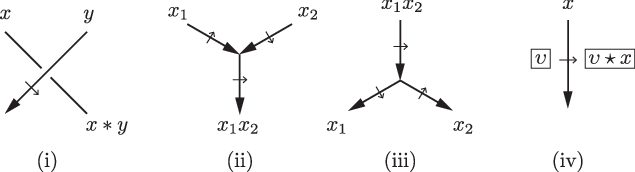}}
	\caption{Rules of a coloring, where $\upsilon\in\Wai$, $x,y\in{X}$ and ${x}_1,{x}_2\in{G}_{\lambda}$.}
	\label{FIG:X_Y_coloring}
	\end{figure}
	
\begin{proposition}\label{prop:coloring_invariant}
Let $X=\bigsqcup_{\lambda\in\Lambda}G_\lambda$ be an \mgr{} and let $\Wai$ be an $X$-set.
Suppose that $D$ is a \diag,
and that $D'$ is a \diag{} obtained by applying one of \Reidemeister{} to $D$ once.
For any $X_{\Wai}$-coloring $\si$ of $D$, there is a unique $X_{\Wai}$-coloring $\si'$ of $D'$ which coincides with $c$ except in the disk where the move is applied.
\end{proposition}

\begin{proof}
Using the axiom of $\Wai$-sets, colors of regions are uniquely determined by those of arcs and one region.
For any $X$-coloring $\si$ of $D$, there is a unique $X$-coloring of $D'$ which coincides with $\si$ except in the disk where the move is applied (Proposition~4.2 in~\cite{IMM20}).
Therefore, the claim follows.
\end{proof}

By Theorem~\ref{theorem_oriented_surface_with_S1} and Proposition~\ref{prop:coloring_invariant},
the cardinality $\#\col_X(D)_{\Wai}$ is an invariant of $(F, o)$,
called the \textit{$X_Y$-coloring number} of $(F, o)$, and denoted by $\operatorname{col}_X(F,o)_\Wai$.
Unfortunately, this invariant is not essential since we have the equality $\#\col_X(D)_{\Wai}=\#\Wai\cdot\#\col_X(D)$.
Proposition~\ref{prop:coloring_invariant} is used in Proposition~\ref{prop:Y-orientation} to define a shadow \mgr{} cocycle invariant of $(F,o)$.

\begin{proposition}\label{prop:coloring_not_distinguish}
For an \surf{} $(F,o)$, it follows that
$\operatorname{col}_X(F,o)_{\Wai}=\operatorname{col}_X(-F^*,-o^*)_{\Wai}$,
where $X$ is an \mgr{} and $\Wai$ is an $X$-set.
\end{proposition}

\begin{proof}
Let $D$ be a \diag{} of $(F,o)$.
We may assume that $D\subset\mathbb R^2\subset S^2$.
Then $D^*:=r(D)$ is the \diag{} of $(F^*,o^*)$ for the map $r:\R^2\to\R^2;(x,y)\mapsto (-x,y)$.
The \diag{} $-D^*$ obtained by reversing the orientation of arcs of $D^*$ represents $(-F^*,-o^*)$.
For any $X_Y$-coloring $c$ of $D$, we can obtain the symmetric $X_Y$-coloring $c^*$ of $-D^*$ as illustrated in Fig.~\ref{FIG:X_Y_coloring_proof_mirror}.
Then $\#\col_X(D)_{\Wai}=\#\col_X(-D^*)_{\Wai}$.
\end{proof}

	\begin{figure}[htpb]\centerline{
	\includegraphics{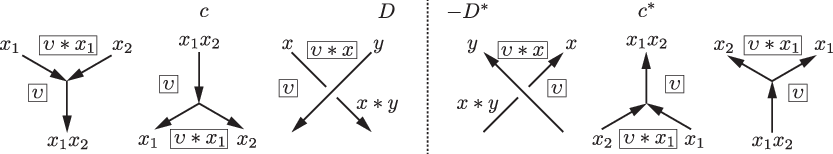}}
	\caption{An $X_Y$-colorings of $D$ and that of $-D^*$.}
	\label{FIG:X_Y_coloring_proof_mirror}	
	\end{figure}

We will distinguish $(F,o)$ and $(-F^*,-o^*)$ by using a cocycle invariant, introduced in Section~\ref{sec:homology}, for some \surf{} $(F,o)$.

\section{(Co)homology of an \mgr{} and shadow cocycle invariants of \surfs}\label{sec:homology}
We denote by $\Z[{X}]$ the free abelian group generated by a non-empty set ${X}$.
For a rack ${Q}$ and a ${Q}$-set $\Wai$, we define a free abelian group $C_n^{\mathrm{R}}({Q})_{\Wai}$ by
\[
C_n^{\mathrm{R}}({Q})_{\Wai}=
\begin{cases}
\displaystyle
\Z\left[
\Wai\times{Q}^n
\right]&(n\in\Z_{\ge0}),\\
0&(n\in\Z_{<0}).
\end{cases}
\]
We represent any element $\left(\upsilon;{a}_1,\ldots,{a}_n\right)\in\Wai\times{Q}^n$ by the noncommutative multiplication form $\la\upsilon\ra\la{a}_1\ra\cdots\la{a}_n\ra$.
We define the boundary homomorphism $\partial_{n}:C_{n}^{\mathrm{R}}({Q})_{\Wai}\to{C}_{n-1}^{\mathrm{R}}({Q})_{\Wai}$ by
\begin{align*}
\partial_{n}\left(\la\upsilon\ra\la{a}_1\ra\cdots\la {a}_{{n}}\ra\right)
&=\sum_{i=1}^{{n}}(-1)^{i}
\la\upsilon\ra\la{a}_1\ra\cdots\la{a}_{i-1}\ra\la{a}_{i+1}\ra\cdots\la{a}_{n}\ra\\
&-\sum_{i=1}^{{n}}(-1)^{i}
\la\upsilon\star{a}_{i}\ra\la{a}_1*{a}_{i}\ra\cdots\la{a}_{i-1}*{a}_{i}\ra\la{a}_{i+1}\ra\cdots\la{a}_{n}\ra
\end{align*}
for any generator $\la\upsilon\ra\la{a}_1\ra\cdots\la {a}_{n}\ra\in{C}_n({Q})_{\Wai}$.
Then $C_*^{\mathrm{R}}({Q})_{\Wai}:=(C_n^{\mathrm{R}}({Q})_{\Wai}, \partial_n)_{n\in\Z}$ is a chain complex, which is the \textit{rack chain complex} of $(Q,\Wai)$ (\cite{CJKLS1999, CJKLS2003, CJKS2001}).
For any abelian group $A$, the cochain complex $C^*_{\mathrm{R}}({Q};{A})_{\Wai}:=(C^n_{\mathrm{R}}({Q};{A})_{\Wai},\delta_n)_{n\in\Z}$ is defined in the ordinary way,
where $C^n_{\mathrm{R}}({Q};{A})_{\Wai}=\operatorname{Hom}_{\Z}(C_n^{\mathrm{R}}({Q})_{\Wai};A)$, and where the coboundary homomorphism $\delta_n$ sends a cochain $h\in{C}^n_{\mathrm{R}}({Q};{A})_{\Wai}$ into $h\circ\partial_{n+1}\in{C}^{n+1}_{\mathrm{R}}({Q};{A})_{\Wai}$.
An $n$-cocycle of $C^*_{\mathrm{R}}({Q};{A})_{\Wai}$ is called a \textit{shadow rack $n$-cocycle} of $(Q,\Wai)$.
If $Y$ is trivial, we often abbreviate $\Wai$ and $\upsilon\in\Wai$:
we write $C_*^{\mathrm{R}}(Q)_{\Wai}$ by $C_*^{\mathrm{R}}(Q)$,
$C^*_{\mathrm{R}}(Q;A)_{\Wai}$ by $C^*_{\mathrm{R}}(Q;A)$ and
$\la\upsilon\ra\la a_1\ra\cdots\la a_n\ra\in C_n(Q)_Y$ by $\la a_1\ra\cdots\la a_n\ra\in C_n(Q)$.
We often call an element of $C^n_{\mathrm{R}}(Q;A)$ a \textit{rack $n$-cocycle} of $Q$ simply.

\begin{remark}\label{REM:iikae_ccycle}{\rm
A rack $Q$ is a $Q$-set with $a\star b=a\ast b$ for any $a,b\in Q$.
We can identify an $n$-cocycle of $C^{n}_{\mathrm{R}}({Q};A)_Q$ with an $(n+1)$-cocycle of $C^{n+1}_{\mathrm{R}}({Q};A)$, since $C_{n}^{\mathrm{R}}(Q)_Q=\Z[Q\times Q^{n}]$ and $C_{n+1}^{\mathrm{R}}({Q})=\Z\left[\{\upsilon\}\times Q^{n+1}\right]$.
For example, a Mochizuki's 3-cocycle~\cite{M2005},
which is a 3-cocycle of $C^{3}_{\mathrm{R}}(R_p;\Z_p)$ of the dihedral quandle $R_p$ with the trivial $R_p$-set,
is identified with a 2-cocycle of $C^{2}_{\mathrm{R}}(R_p;\Z_p)_{R_p}$.
}\end{remark}

For an \mgr{} ${X}=\bigsqcup_{\lambda\in\Lambda}G_{\lambda}$ and an $X$-set $\Wai$,
we define a free abelian group $C_n(X)_{\Wai}$ by
\[
C_n({X})_{\Wai}=
\begin{cases}
\displaystyle
\Z\left[
\bigsqcup_{n_1+\cdots+n_k=n}
\Wai\times\prod_{i=1}^{k}
\biggl(\bigsqcup_{\lambda\in\Lambda}{G_\lambda}^{{n}_i}\biggr)
\right]&(n\in\Z_{\ge0}),\\
0      &(n\in\Z_{<0}).
\end{cases}
\]
Here, the indices $n_1,\ldots,n_k$ in the disjoint union run over all the partitions of $n$.
We represent
$\left(\upsilon;x_{1,1},\ldots,{x}_{1,n_1};\ldots;{x}_{k,1},\ldots,{x}_{k,n_k}\right)\in\Wai\times\bigsqcup_{\lambda\in\Lambda}{{G}_\lambda}^{n_1}\times\cdots\times\bigsqcup_{\lambda\in\Lambda}{{G}_\lambda}^{n_k}
$ by the noncommutative multiplication form $\la\upsilon\ra\la{x}_{1,1},\ldots,{x}_{1,n_1}\ra\cdots\la{x}_{k,1},\ldots,{x}_{k,n_k}\ra$.
For example,
\begin{align*}
  C_3(X)_{\Wai}
&=\Z\Biggl[
  \Bigl(\Wai
  \times\bigsqcup_{\lambda\in\Lambda}{G_\lambda}
  \times\bigsqcup_{\lambda\in\Lambda}{G_\lambda}
  \times\bigsqcup_{\lambda\in\Lambda}{G_\lambda}\Bigr)
  \sqcup
  \Bigl(\Wai
  \times\bigsqcup_{\lambda\in\Lambda}{G_\lambda}
  \times\bigsqcup_{\lambda\in\Lambda}{G_\lambda}^2\Bigr)\\
& \hspace{2em}\sqcup
  \Bigl(\Wai
  \times\bigsqcup_{\lambda\in\Lambda}{G_\lambda}^2
  \times\bigsqcup_{\lambda\in\Lambda}{G_\lambda}\Bigr)
  \sqcup
  \Bigl(\Wai
  \times\bigsqcup_{\lambda\in\Lambda}{G_\lambda}^3\Bigr)
  \Biggr]\\
&=\Z\Biggl[
  \Set{\la\upsilon\ra\la{x}\ra\la{y}\ra\la{z}\ra|
  \begin{array}{l}
  \text{$\upsilon\in\Wai$,}\\
  \text{${x},{y},{z}\in{X}$}
  \end{array}}
  \sqcup
  \Set{\la\upsilon\ra\la{x}\ra\la{y}_1,{y}_2\ra|
  \begin{array}{l}
  \text{$\upsilon\in\Wai$, ${x}\in{X}$}\\
  \text{$\lambda\in\Lambda$, ${y}_1,{y}_2\in{G}_{\lambda}$}
  \end{array}}\\
& \hspace{2em}\sqcup
  \Set{\la\upsilon\ra\la{x}_1,{x}_2\ra\la{y}\ra|
  \begin{array}{l}
  \text{$\upsilon\in\Wai$, ${y}\in{X}$,}\\
  \text{$\lambda\in\Lambda$, ${x}_1,{x}_2\in{G}_{\lambda}$}
  \end{array}}
  \sqcup
  \Set{\la\upsilon\ra\la{x}_1,{x}_2,{x}_3\ra|
  \begin{array}{l}
  \text{$\upsilon\in\Wai$, $\lambda\in\Lambda$,}\\
  \text{${x}_1,{x}_2,{x}_3\in{G}_{\lambda}$}
  \end{array}}
  \Biggr].
\end{align*}
\par
Let $\bm{x}$ be a sequence $x_1,\ldots,{x}_m$ of elements of $G_\lambda$ and let $y\in X$.
We denote by $\bm{x}*y$ the sequence $x_1*y,\ldots,{x}_m*y$.
The notation $\la\upsilon\ra\la\bm{x}_1\ra\cdots\la\bm{x}_k\ra*y$ means $\la\upsilon\star{y}\ra\la\bm{x}_1*y\ra\cdots\la\bm{x}_k*y\ra$ for $\la\upsilon\ra\la\bm{x}_1\ra\cdots\la\bm{x}_k\ra\in\Wai\times\bigsqcup{G}_\lambda^{n_1}\times\cdots\times\bigsqcup{G}_\lambda^{n_k}$.
We denote by $|\bm{x}|$ the length of $\bm{x}$, that is, $|\bm{x}|=m$ and set $|\la\upsilon\ra\la\bm{x}_1\ra\cdots\la\bm{x}_k\ra|=|\bm{x}_1|+\cdots+|\bm{x}_k|$.
We define $0*y=0$ for $0\in C_n({X})_{\Wai}$.
The notation
$
\la\upsilon\ra\la\bm{x}_{1}\ra\cdots\la\bm{x}_{{p}}\ra
\bigl(
*{w}\la\bm{y}_{0}\ra+\sum_{{i}=1}^{{q}}(-1)^{{i}}\la\bm{y}_{{i}}\ra
\bigr)
\la\bm{z}_{1}^{}\ra\cdots\la\bm{z}_{{r}}^{}\ra
$
means the element 
$
\left(\la\upsilon\ra\la\bm{x}_{1}\ra\cdots\la\bm{x}_{{p}}\ra*{w}\right)
\la\bm{y}_{0}\ra\la\bm{z}_{1}^{}\ra\cdots\la\bm{z}_{{r}}^{}\ra+\sum_{{i}=1}^{{q}}(-1)^{{i}}\la\upsilon\ra\la\bm{x}_{1}\ra\cdots\la\bm{x}_{{p}}\ra\la\bm{y}_{{i}}\ra
\la\bm{z}_{1}^{}\ra\cdots\la\bm{z}_{{r}}^{}\ra
$,
where $|\bm{y}_0|=|\bm{y}_i|$ for any $i$.
The notation
$
\la\upsilon\ra\la\bm{x}_{1}\ra\cdots\la\bm{x}_{{p}}\ra
\left(*{w}\la\ra-\la\ra\right)
\la\bm{z}_{1}^{}\ra\cdots\la\bm{z}_{{r}}^{}\ra
$
means the element 
$
(\la\upsilon\ra\la\bm{x}_{1}\ra\cdots\la\bm{x}_{{p}}\ra*{w})
\la\bm{z}_{1}^{}\ra\cdots\la\bm{z}_{{r}}^{}\ra
-\la\upsilon\ra\la\bm{x}_{1}\ra\cdots\la\bm{x}_{{p}}\ra\la\bm{z}_{1}^{}\ra\cdots\la\bm{z}_{{r}}^{}\ra.
$
\par
For $\la\bm{x}\ra=\la{x}_1,\ldots,{x}_m\ra$,
where ${x}_1,\ldots,{x}_m\in G_\lambda$,
we define an operator $\PT\la\bm{x}\ra$ by
\[
\PT\la\bm{x}\ra
=*{x}_1\la\bm{x}^{0}\ra+\displaystyle\sum_{{i}=1}^{{m}}(-1)^{{i}}\la\bm{x}^{i}\ra.
\]
Here, $\bm{x}^{i}$ is the sequence ${x}_1,\ldots,x_ix_{i+1},\ldots,{x}_m$ for any $i$ with $0<i<m$, $\bm{x}^{0}$ is the sequence ${x}_2,\ldots,{x}_m$, and $\bm{x}^{m}$ is the sequence ${x}_1,\ldots,{x}_{m-1}$.
We set $\la\bm{x}^{0}\ra=\la\bm{x}^{i}\ra=\la\ra$ if $m=1$.
\par
We define the boundary homomorphism $\partial_n:C_n(X)_{\Wai}\to C_{n-1}(X)_{\Wai}$ by
\[
\partial_n\left(\la\upsilon\ra\la\bm{x}_1\ra\cdots\la \bm{{x}}_{{k}}\ra\right)
=\begin{cases}
\displaystyle\sum_{i=1}^{{k}}(-1)^{|\la\upsilon\ra\la\bm{{x}}_1\ra\cdots\la\bm{{x}}_{{i-1}}\ra|}
\la\upsilon\ra\la\bm{{x}}_1\ra\cdots\PT\la\bm{{x}}_i \ra\cdots\la\bm{{x}}_{{k}}\ra&(n\in\Z_{>0}),\\
0& (n\in\Z_{\le0}).
\end{cases}
\]
where $\la\upsilon\ra\la\bm{x}_1\ra\cdots\la\bm{{x}}_{{k}}\ra$ is any generator of ${C}_n({X})_{\Wai}$, $n=|\la\upsilon\ra\la\bm{x}_1\ra\cdots\la\bm{{x}}_{{k}}\ra|$,
and where $|\la\upsilon\ra\la\bm{{x}}_1\ra\cdots\la\bm{{x}}_{{i-1}}\ra|=0$ when $i=1$.
We often write $\partial_n$ by $\partial$.

\begin{example}\label{example_partial_3}
For any generator $\la\upsilon\ra\la{x}\ra\la{y}\ra$ of $C_2({X})_{\Wai}$,
\begin{align*}
  \partial_2(\la\upsilon\ra\la{x}\ra\la{y}\ra)
&=\la\upsilon\ra\widetilde{\partial}\la{x}\ra\la{y}\ra
+(-1)^1\la\upsilon\ra\la{x}\ra\widetilde{\partial}\la{y}\ra\\
&=\la\upsilon\ra(*{x}\la\ra-\la\ra)\la y\ra
  -\la\upsilon\ra\la{x}\ra(*{y}\la\ra-\la\ra)\\
&=\la\upsilon\star{x}\ra\la{y}\ra
 +\la\upsilon\ra\la{x}\ra
 -\la\upsilon\ra\la{y}\ra
 -\la\upsilon\star{y}\ra\la{x}*{y}\ra.
\end{align*}
For any generator $\la\upsilon\ra\la{x}_{1},{x}_{2}\ra$ of $C_2({X})_{\Wai}$,
\begin{align*}
  \partial_2(\la\upsilon\ra\la{x}_{1},{x}_{2}\ra)
&=\la\upsilon\ra\widetilde{\partial}\la {x}_{1},{x}_{2}\ra\\
&=\la\upsilon\ra(*{x}_1\la{x}_2\ra-\la{x}_1{x}_2\ra+\la{x}_1\ra)\\
&=\la\upsilon\star{x}_1\ra\la{x}_2\ra
 +\la\upsilon\ra\la{x}_1 \ra
 -\la\upsilon\ra\la{x}_1{x}_2\ra.
\end{align*}
For any generator $\la\upsilon\ra\la{x}\ra\la{y}\ra\la{z}\ra$ of $C_3({X})_{\Wai}$,
\begin{align*}
\partial_3(\la\upsilon\ra\la{x}\ra\la{y}\ra\la{z}\ra)
&= \la\upsilon\ra\widetilde{\partial}\la{x}\ra\la{y}\ra\la{z}\ra
  +(-1)^1\la\upsilon\ra\la{x}\ra\widetilde{\partial}\la{y}\ra\la{z}\ra
  +(-1)^2\la\upsilon\ra\la{x}\ra\la{y}\ra\widetilde{\partial}\la{z}\ra\\
&=\la\upsilon\star{x}\ra\la{y}\ra\la{z}\ra
  +\la\upsilon\ra\la{x}\ra\la{z}\ra
  +\la\upsilon\star{z}\ra\la{x}*{z}\ra\la{y}*{z}\ra
  -\la\upsilon\ra\la{y}\ra\la{z}\ra
  -\la\upsilon\star{y}\ra\la{x}*{y}\ra\la{z}\ra
  -\la\upsilon\ra\la{x}\ra\la{y}\ra.
\end{align*}
For any generator $\la\upsilon\ra\la{x}\ra\la{y}_1,{y}_2\ra$ of $C_3({X})_{\Wai}$,
\begin{align*}
\partial_3(\la\upsilon\ra\la{x}\ra\la{y}_1,{y}_2\ra)
&=\la\upsilon\ra\widetilde{\partial}\la{x}\ra\la{y}_1,{y}_2\ra
  +(-1)^1\la\upsilon\ra\la{x}\ra\widetilde{\partial}\la{y}_1,{y}_2\ra\\
&=\la\upsilon\star{x}\ra\la{y}_1,{y}_2\ra
  +\la\upsilon\ra\la{x}\ra\la{y}_1{y}_2\ra
  -\la\upsilon\ra\la{y}_1,{y}_2\ra
  -\la\upsilon\star{y}_1\ra\la{x}*{y}_1\ra\la{y}_2\ra
  -\la\upsilon\ra\la{x}\ra\la{y}_1\ra.
\end{align*}
For any generator $\la\upsilon\ra\la{x}_1,{x}_2\ra\la{y}\ra$ of $C_3({X})_{\Wai}$,
\begin{align*}
\partial_3(\la\upsilon\ra\la{x}_1,{x}_2\ra\la{y}\ra)
&=\la\upsilon\ra\widetilde{\partial}\la{x}_1,{x}_2\ra\la{y}\ra
  +(-1)^2\la\upsilon\ra\la{x}_1,{x}_2\ra\widetilde{\partial}\la{y}\ra\\
&=\la\upsilon\star{x}_1\ra\la{x}_2\ra\la{y}\ra
  +\la\upsilon\ra\la{x}_1\ra\la y\ra
  +\la\upsilon\star{y}\ra\la{x}_1*{y},{x}_2*{y}\ra
  -\la\upsilon\ra\la{x}_1{x}_2\ra\la{y}\ra
  -\la\upsilon\ra\la{x}_1,{x}_2\ra.
\end{align*}
For any generator $\la\upsilon\ra\la{x}_1,{x}_2,{x}_3\ra$ of $C_3({X})_{\Wai}$,
\begin{align*}
\partial_3(\la\upsilon\ra\la{x}_1,{x}_2,{x}_3\ra)
&=\la\upsilon\ra\widetilde{\partial}\la{x}_1,{x}_2,{x}_3\ra\\
&=\la\upsilon\star{x}_1\ra\la{x}_2,{x}_3\ra
+\la\upsilon\ra\la{x}_1,{x}_2{x}_3\ra
-\la\upsilon\ra\la{x}_1{x}_2,{x}_3\ra
-\la\upsilon\ra\la{x}_1,{x}_2\ra.
\end{align*}
\end{example}

\begin{remark}\label{REM_leibniz}{\rm
When $n\ge1$, the boundary homomorphism $\partial_{n}$ is defined inductively on $k$ as follows.
We can check that easily.
\begin{itemize}
\item
$\partial_{{n}}\left(\la\upsilon\ra\la\bm{x}_1\ra\right)
=\la\upsilon\ra\widetilde{\partial}\la\bm{x}_1\ra$,
where ${n}=|\bm{x}_1|$.
\item
$\partial_{{n}}\left(\la\upsilon\ra\la\bm{x}_1\ra\cdots\la\bm{x}_{k+1}\ra\right)
=\partial_{{m}}\left(\la y\ra\la\bm{x}_1\ra\cdots\la\bm{x}_k\ra\right)\la\bm{x}_{{k}+1}\ra
+(-1)^{{m}}\la\upsilon\ra\la\bm{x}_1\ra\cdots\la\bm{x}_k\ra\widetilde{\partial}\la\bm{x}_{k+1}\ra$,
where $m=|\la\upsilon\ra\la\bm{x}_{1}\ra\cdots\la\bm{x}_{k}\ra|$ and
${n}=m+|\bm{x}_{{k}+1}|$ (Leibniz rule).
\end{itemize}
}\end{remark}

The following lemma is an essential part of the proof of Proposition~\ref{prop_partialpartialzero}.
\begin{lemma}\label{LEM_zeromap}
Let $w=\la\upsilon\ra\la\bm{x}_1\ra\cdots \la\bm{x}_k\ra$ be a generator of ${C}_{{n}}(X)_{\Wai}$ for $n\ge0$.
Then, we have the followings.
\begin{enumerate}
\item
$\partial\left({w}*{x}\right)=\partial({w})*{x}$ for any $x\in X$.
\item
$\partial({w}\PT\la\bm{x}\ra)=\partial({w})\PT\la\bm{x}\ra$ for a sequence $\bm{x}={x}_1,\ldots,{x}_m$ of elements of $G_\lambda$.
\end{enumerate}
\end{lemma}
\begin{proof}
(1)
If $k=0$, the equation clearly holds.
Then we assume that $k\ge1$.
This statement follows from the equation:
\begin{align*}
&\phantom{=}\la\upsilon\star x\ra\la\bm{x}_1*x\ra\cdots\PT\la\bm{x}_i*x\ra\cdots\la\bm{x}_k*x\ra\\
&=\la\upsilon\star x\ra\la\bm{x}_1*x\ra\cdots\bigl(*(x_{i1}*x)\la{\bm{x}_i}^0*x\ra+
\sum_{j=1}^{|\bm{x}_i|}(-1)^j
\la{\bm{x}_i}^j*x\ra\bigr)\cdots\la\bm{x}_k*x\ra\\
&=\Bigl(\la\upsilon\ra\la\bm{x}_1\ra\cdots\bigl(*x_{i1}\la{\bm{x}_i}^0\ra+
\sum_{j=1}^{|\bm{x}_i|}(-1)^j
\la{\bm{x}_i}^j\ra\bigr)\cdots\la\bm{x}_k\ra\Bigl)*x\\
&=(\la\upsilon\ra\la\bm{x}_1\ra\cdots\PT\la\bm{x}_i\ra\cdots\la\bm{x}_k\ra)*x
\end{align*}
for any $i$ with $1\le i\le k$,
where $x_{i1}$ means the first element of $\bm{x}_i$, and where we note $\la{\bm{x}_i}^j*x\ra=\la{\bm{x}_i}^0*x\ra=\la\ra$ if $|\bm{x}_i|=1$.
The second equation above follows from the axiom (2) of \mgr{s} and the axiom (2) of $X$-sets.
\par
(2)
Suppose that $m=1$.
We have
$\partial\bigl({w}\PT\la{x}_1\ra\bigr)
=\partial\left({w}*{x}_1-{w}\right)
=\partial({w})*{x}_1-\partial({w})
=\partial({w})\PT\la{x}_1\ra$,
where the second equation follows from Lemma~\ref{LEM_zeromap} (1).

We assume that $m\ge2$.
We show the claim $(*)$ ${w}\,\PT\PT\la\bm{x}\ra=0$, where we set
$w\PT\PT\la\bm{x}\ra
=\left({w}*{x}_1\right)\PT\la\bm{x}^0\ra+\sum_{i=1}^{m}(-1)^iw\,\PT\la\bm{x}^i\ra$.
We put $\la\bm{x}^{ij}\ra=\la\left(\bm{x}^i\right)^j\ra$ for any $i,j\in\Z_{\ge0}$ with $i\le m$ and $j<m$.
Then,
{\normalsize
\begin{align*}
{w}\PT\PT\la\bm{x}\ra
&=({w}*{x}_1)
\bigl(*x_2\la\bm{x}^{00}\ra+\sum_{j=1}^{m-1}(-1)^j\la\bm{x}^{0j}\ra\bigr)
+\sum_{i=1}^{m}\Bigl((-1)^i{w}
\bigl(*x_1^i\la\bm{x}^{i0}\ra+\sum_{j=1}^{m-1}(-1)^j\la\bm{x}^{ij}\ra\bigr)\Bigr)\\
&=\ul{\Bigl(w*x_1*x_2\la\bm{x}^{00}\ra+\sum_{j=1}^{m-1}(-1)^j\,w*x_1\la\bm{x}^{0j}\ra\Bigr)+\Bigl(-w*x_1x_2\la\bm{x}^{10}\ra+\sum_{i=2}^{m}(-1)^i\,w*x_1\la\bm{x}^{i0}\ra\Bigr)}_{(A)}\\
&\phantom{=}+\sum_{i=1}^{m-1}\ul{\Bigl(\sum_{j=i}^{m-1}(-1)^{i+j}\,w\la\bm{x}^{ij}\ra+\sum_{j=i}^{m-1}(-1)^{(j+1)+i}w\la\bm{x}^{(j+1)i}\ra\Bigr)}_{(B)}.
\end{align*}
}
If $i\le j$, $\la\bm{x}^{ij}\ra=\la\bm{x}^{(j+1)i}\ra$ by the definition of $\la\bm{x}^{ij}\ra$.
Then, $(A)=(w*x_1*x_2-w*x_1x_2)\la\bm{x}^{00}\ra$ and $(B)=0$.
By the axiom (1) of \mgr{s} and the axiom (1) of $X$-sets, $(A)=0$.
Therefore, ${w}\,\PT\PT\la\bm{x}\ra=0$. 
Then,
\begin{align*}
\partial({w}\PT\la\bm{x}\ra)
&=\partial\left({w}*{x}_1{\la\bm{x}^0\ra}\right)
  +\sum_{{i}=1}^{{m}}(-1)^i\partial\left({w}{\la\bm{x}^i\ra}\right)\\
&=\left(\partial({w}*{x}_1){\la\bm{x}^0\ra}+(-1)^{{n}}({w}*{x}_1)\widetilde{\partial}{\la\bm{x}^0\ra}\right)
+\sum_{{i}=1}^{{m}}(-1)^{{i}}\left(\partial({w}){\la\bm{x}^i\ra}
+(-1)^n{w}\widetilde{\partial}{\la\bm{x}^i\ra}\right)\\
&=\partial({w})*x_1{\la\bm{x}^0\ra}+\sum_{{i}=1}^{{m}}(-1)^{{i}}\partial({w}){\la\bm{x}^i\ra}+(-1)^{n}{w}\widetilde{\partial}\PT\la\bm{x}\ra\\
&=\partial({w})\PT\la\bm{x}\ra.
\end{align*}
Here, the second equation follows from Remark~\ref{REM_leibniz},
the third equation follows from Lemma~\ref{LEM_zeromap} (1),
and the fourth equation follows from the claim $(*)$  above.
\end{proof}

\begin{proposition}\label{prop_partialpartialzero}
$(C_n(X)_{\Wai},\partial_n)_{n \in\Z}$ is a chain complex.
\end{proposition}

\begin{proof}
It suffices to show $\partial\circ\partial\left({w}\right)=0$ for any generator $w={\la\upsilon\ra\la\bm{x}_1\ra\cdots\la\bm{x}_{{k}}\ra}$ of $C_{{n}}(X)_\Wai$.
The equation clearly holds if $n\le1$.
Then we assume that $n\ge2$.
We show the equation by an induction on $k$.
\par
In the case where $k=1$, we have 
$\partial\circ\partial\left(w\right)
=\partial(\la\upsilon\ra\widetilde{\partial}\la\bm{x}_1\ra)
=\partial(\la\upsilon\ra)\PT\la\bm{x}_1\ra
=0$,
where the second equation follows from Lemma~\ref{LEM_zeromap} (2).
\par
We assume that $\partial\circ\partial({w})=0$;
we show $\partial\circ\partial({w}\la\bm{x}_{k+1}\ra)=0$.
\begin{align*}
\partial\circ\partial({w}\la\bm{x}_{k+1}\ra)
&=\partial\left(\partial({w})\la\bm{x}_{{k}+1}\ra
+(-1)^{{n}}{w}\widetilde{\partial}\la\bm{x}_{{k}+1}\ra\right)\\
&=\left(\partial(\partial({w}))\la\bm{x}_{{k}+1}\ra
+(-1)^{{n}-1}\partial({w})\widetilde{\partial}\la\bm{x}_{{k}+1}\ra\right)
+(-1)^{{n}}\partial\left({w}\widetilde{\partial}\la\bm{x}_{{k}+1}\ra\right)\\
&=\partial\circ\partial({w})\la\bm{x}_{{k}+1}\ra\\
&=0.
\end{align*}
Here, the first and second equations follow from Remark~\ref{REM_leibniz},
the third equation follows from Lemma~\ref{LEM_zeromap} (2),
and fourth equation follows from the assumption of the induction.
\end{proof}

We call $C_*(X)_{\Wai}:=(C_n(X)_{\Wai},\partial_n)_{n \in\Z}$ the \textit{\mgr{} chain complex} of $(X,\Wai)$.
We denote by $H_n(X)_{\Wai}$ the $n$-th homology group of $C_*(X)_{\Wai}$.
For an abelian group $A$, the cochain complex $C^*({X};{A})_{\Wai}:=(C^n({X};{A})_{\Wai},\delta_n)_{n\in\Z}$ is defined in the ordinary way,
where $C^n({X};A)_{\Wai}=\operatorname{Hom}_{\Z}(C_n({X})_{\Wai};A)$, and where the coboundary homomorphism $\delta_n$ sends a cochain $h\in{C}^n({X};{A})_{\Wai}$ into $h\circ\partial_{n+1}\in{C}^{n+1}({X};{A})_{\Wai}$.
An $n$-cocycle of $C^*(X;{A})_{\Wai}$ is called a \textit{shadow \mgr{} $n$-cocycle} of $(X,\Wai)$.
If $Y$ is trivial, we often abbreviate $\Wai$ and $\upsilon\in\Wai$:
we write $C_*(X)_{\Wai}$ by $C_*(X)$, $C^*(X;A)_{\Wai}$ by $C^*(X;A)$ and $\la\upsilon\ra\la\bm{x}_1\ra\cdots\la\bm{x}_k\ra\in C_n(X)_Y$ by $\la\bm{x}_1\ra\cdots\la\bm{x}_k\ra\in C_n(X)$.
We often call an element of $C^n(X;A)$ an \textit{\mgr{} $n$-cocycle} of $X$ simply.
\par
For an $X_\Wai$-coloring $\si$ of a \diag{} $D$, we define the \textit{local weight} $w(\chi;\si)\in{C}_2(X)_{\Wai}$ at any $\chi\in U(D)\sqcup V(D)$ as depicted in Fig.~\ref{fig:local_chain},
where $U(D)$ (resp.~$V(D)$) is the set of crossings (resp.~vertices) of $D$.

	\begin{figure}[htpb]\centerline{
	\includegraphics{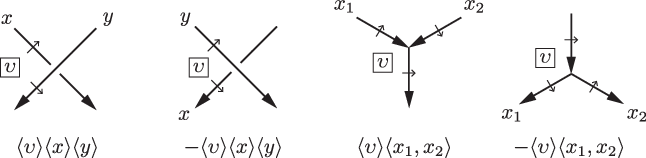}}
	\caption{Local wights for a crossing or a vertex $\chi$.}
	\label{fig:local_chain}
	\end{figure}

We define a 2-chain $W(D;\si)\in{C}_2(X)_{\Wai}$ by
\[W(D;\si)=\sum_{\chi\in U(D)\sqcup V(D)}w(\chi;\si).\]

\begin{lemma}\label{LEM_main_mgr}
Let $X$ be an \mgr{} and $\Wai$ be an $X$-set.
For any $X_{\Wai}$-coloring $\si$ of a \diag{} $D$, the 2-chain $W(D;\si)$ is a 2-cycle.
\end{lemma}

\begin{proof}
A connected component of $D\setminus U(D)\setminus V(D)$ is called a \textit{semiarc}.
Let $\chi^+$ (resp.~$\chi^-$) be a positive (resp.~negative) crossing of $D$,
and let $\chi^+_i$ (resp.~$\chi^-_i$) be the semiarc incident with $\chi^+$ (resp.~$\chi^-$) as illustrated in Fig.~\ref{fig:prove_2cycle} ($i=1,2,3,4$). 
Let $\omega^+$ (resp.~$\omega^-$) be a vertex which is the initial (resp.~terminal) point of exactly one edge, and let $\omega^+_i$ (resp.~$\omega^-_i$) be the semiarc incident with $\omega^+$ (resp.~$\omega^-_i$) as illustrated in Fig.~\ref{fig:prove_2cycle} ($i=1,2,3$).
There is the unique region $\rho_s$ facing a semiarc $s$ so that the normal orientation of $s$ points from the region $\rho_s$ to the opposite region with respect to $s$.

	\begin{figure}[htpb]\centerline{
	\includegraphics{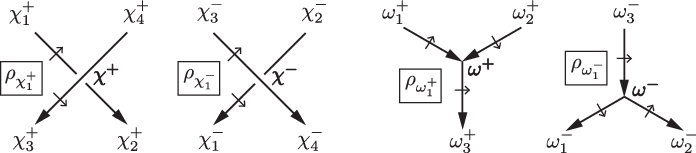}}
	\caption{Crossings $\chi^+$, $\chi^-$ and vertices $\omega^+$, $\omega^-$.}
	\label{fig:prove_2cycle}
	\end{figure}

For an $X_\Wai$-coloring $c$ of $D$, we define $\si(s)=\si(a)$ if an arc $a$ of $D$ contains a semiarc $s$.
Moreover, we write $s=\si(s)$ for simplicity notation.
Then, we immediately have the following equations $(*)$:
\begin{align*}
&\rho_{\chi^\pm_1}=\rho_{\chi^\pm_3},\quad
\chi^\pm_3=\chi^\pm_4,\quad
\rho_{\omega^\pm_1}=\rho_{\omega^\pm_3},\quad
\rho_{\chi^\pm_1}\star \chi^\pm_1=\rho_{\chi^\pm_4},\quad
\rho_{\chi^\pm_1}\star \chi^\pm_3=\rho_{\chi^\pm_2},\quad\\
&\chi^\pm_1\ast\chi^\pm_3=\chi^\pm_2,\quad
\rho_{\omega^\pm_1}\star\omega^\pm_1=\rho_{\omega^\pm_2},\quad
\omega^\pm_1\omega^\pm_2=\omega^\pm_3.
\end{align*}
For a crossing or a vertex $\chi$ of $D$ and a semiarc $s$ incident to $\chi$, we denote by $\mathcal S(D;\chi)$ the set of semiarcs incident to $\chi$, and define $\epsilon(s;\chi)$ by
\[
\epsilon(s;\chi)=
\begin{cases}
 1&\text{if the orientation of $s$ points to $\chi$},\\
-1&\text{otherwise}
\end{cases}.
\]
Then, by the equations $(*)$ and an equation of Example~\ref{example_partial_3},
\begin{align*}
\partial_2(w(\chi^\pm;{\si}))&=\partial_2\bigl(\pm\la\rho_{\chi_1^\pm}\ra\la\chi_1^\pm\ra\la\chi_3^\pm\ra\bigr)\\
&=\pm\bigl(\la\rho_{\chi_1^\pm}\star\chi_1^\pm\ra\la\chi_3^\pm\ra+\la\rho_{\chi_1^\pm}\ra\la\chi_1^\pm\ra-\la\rho_{\chi_1^\pm}\ra\la\chi_3^\pm\ra-\la\rho_{\chi_1^\pm}\star\chi_3^\pm\ra\la\chi_1^\pm*\chi_3^\pm\ra\bigr)\\
&=\pm\la \rho_{\chi_4^\pm}\ra\la \chi_4^\pm\ra\pm\la \rho_{\chi_1^\pm}\ra\la \chi_1^\pm\ra\mp\la \rho_{\chi_3^\pm}\ra\la \chi_3^\pm\ra\mp\la \rho_{\chi_2^\pm}\ra\la \chi_2^\pm\ra\\
&=\sum_{s\in\mathcal S(D;\chi^\pm)}\epsilon(s;\chi^\pm)\la \rho_s\ra\la s\ra,\text{ and}
\end{align*}
\begin{align*}
\partial_2(w(\omega^\pm;\si))
&=\partial_2\bigl(\pm\la \rho_{\omega_1^\pm}\ra\la \omega_1^\pm,\,\omega_2^\pm\ra\bigr)\\
&=\pm\bigl(\la \rho_{\omega_1^\pm}\star\omega_1^\pm\ra\la\omega_2^\pm\ra
+\la\rho_{\omega_1^\pm}\ra\la\omega_1^\pm\ra
-\la\rho_{\omega_1^\pm}\ra\la\omega_1^\pm\omega_2^\pm\ra\bigr)\\
&=\pm
\la\rho_{\omega_2^\pm}\ra\la\omega_2^\pm\ra
\pm\la\rho_{\omega_1^\pm}\ra\la\omega_1^\pm\ra
\mp\la\rho_{\omega_3^\pm}\ra\la\omega_3^\pm\ra\\
&=
\sum_{s\in\mathcal S(D;\omega^\pm)}\epsilon(s;\omega^\pm)\la\rho_s\ra\la s\ra.
\end{align*}
Therefore, we have $\partial_2(W({D};{\si}))=0$ since
\begin{align*}
\partial_2(W({D};{\si}))
&=\sum_{\chi\in{{U}({D})\sqcup{V}(D)}}\partial_2\left({w}(\chi;{\si})\right)\\
&=\sum_{\chi\in{{U}({D})\sqcup{V}(D)}}\Bigl(\sum_{s\in\mathcal S({D};\chi)}\epsilon(s;\chi)\la{\rho}_{s}\ra\la s\ra\Bigr)\\
&=\sum_{s\in\mathcal S({D})}\left(\epsilon(s;s^+)\la{\rho}_s\ra\la s\ra+\epsilon(s;s^-)\la{\rho}_{s}\ra\la s\ra\right),
\end{align*}
where $\mathcal S(D)$ is the set of semiarcs of $D$, and $s^\pm$ are the end points of $s\in \mathcal S(D)$.
\end{proof}

By the lemma above, it holds that
$\theta(W(D;\si))=\theta'(W(D;\si))$ for any cohomologous 2-cocycles $\theta$ and $\theta'$ of $C^2({X};{A})_{\Wai}$.
\par
We denote by $[W]$ the homology class that contains $W$ for a 2-cycle $W\in{C}_2(X)_\Wai$.

\begin{proposition}\label{prop:Y-orientation}
Let $X$ be an \mgr{} and $\Wai$ be an $X$-set.
Suppose that $D$ is a \diag, and that $D'$ is a \diag{} obtained by applying one of \Reidemeister{} to $D$ once.
For any $X_{\Wai}$-coloring $\si$ of $D$, it follows that $\left[W(D;\si)\right]=\left[W(D';\si')\right]$ in $H_2(X)_{\Wai}$, where $\si'$ is the unique $X_{\Wai}$-coloring of $D'$ described in Proposition~\ref{prop:coloring_invariant}
\end{proposition}
\begin{proof}
It suffices to show $W(D;\si)-W(D';\si')=\partial_3(w)$ for some element $w\in{C}_3(X)_\Wai$.
When $D'$ is obtained from $D$ by applying an R2 move, $W(D;\si)-W(D';\si')=0=\partial_3(0)$.
When $D'$ is obtained from $D$ by applying an R3 move, $W(D;\si)-W(D';\si')=\partial_3(\pm\la\upsilon\ra\la{x}\ra\la{y}\ra\la{z}\ra)$ (see Example~\ref{example_partial_3}).
When $D'$ is obtained from $D$ by applying an R5 move, $W(D;\si)-W(D';\si')=\partial_3(\pm\la\upsilon\ra\la{x}\ra\la{y}_1,{y}_2\ra)$ or $\partial_3(\pm\la\upsilon\ra\la{x}_1,{x}_2\ra\la{y}\ra)$.
When $D'$ is obtained from $D$ by applying an R6 move, $W(D;\si)-W(D';\si')=\partial_3(\pm\la\upsilon\ra\la{x}_1,{x}_2,{x}_3\ra)$.
In Fig.~\ref{FIG:spatial_surface_homology_reidemeister}, some examples are written.
\end{proof}

		\begin{figure}[htbp]\centerline{
		\includegraphics{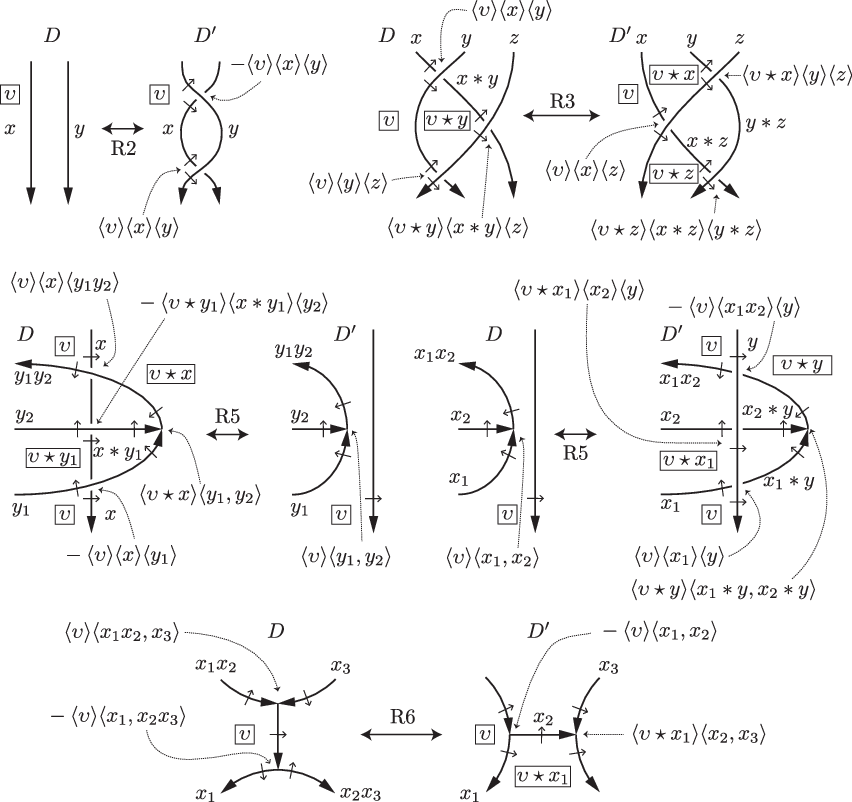}}
		\caption{Local 2-chains of $W(D;\si)$ and $W(D';\si')$.}
		\label{FIG:spatial_surface_homology_reidemeister}
		\end{figure}

For a \diag{} $D$ and a shadow \mgr{} 2-cocycle $\theta:C_2(X)_\Wai\to A$, we define the following multisets:
\begin{align*}
\mathcal{H}({D})&=\Set{[W(D;\si)]\in{H}_2(X)_{\Wai}|\si\in\col_X(D)_{\Wai}},\\
\Phi_\theta({D})&=\Set{\theta(W(D;\si))|\si\in\col_X(D)_{\Wai}},
\end{align*}
where $X$ is an \mgr, $\Wai$ is an $X$-set, and $A$ is an abelian group.
\par
The following theorem follows from Theorem~\ref{theorem_oriented_surface_with_S1} and Proposition~\ref{prop:Y-orientation}.
\begin{theorem}\label{thm:main1}
Let $D$ be a \diag{} of an \surf{} $(F,o)$.
Then, the multisets $\mathcal H(D)$ and $\Phi_\theta(D)$ are invariants of $(F,o)$, respectively.
\end{theorem}

Then we put $\mathcal{H}(F,o)=\mathcal{H}({D})$.
We also put $\Phi_\theta(F,o)=\Phi_\theta({D})$, which is called a \textit{(shadow) \mgr{} cocycle invariant}.
\par
Let $Q\times\Z_{k\type Q}$ be the associated \mgr{} of a rack $Q$ and $Y$ be a $(Q\times\Z_{k\type Q})$-set.
For a $(Q\times\Z_{k\type Q})_Y$-coloring $\si:\mathcal A(D)\cup\mathcal R(D)\to (Q\times\Z_{k\type Q})\cup\Wai$ of a \diag{} $D$, we set $\gcd \si=\gcd\left(\Set{\operatorname{pr}_2\circ\,c(a)|a\in \mathcal A(D)}\cup\{k\type Q\}\right)$, where $\operatorname{pr}_2:Q\times \Z_{k\type Q}\to\Z_{k\type Q}$ is the canonical projection, and where we regard $\operatorname{pr}_2\circ\,c(a)$ as an element of $\Z$ which represents $\operatorname{pr}_2\circ\,c(a)\in\Z_{k\type Q}$.
Here, we set $\gcd (S\cup\{\infty\})=\gcd S$ for a set $S$ of integers.
Then, we have the following lemma immediately.

\begin{lemma}\label{lemma:gcd_new}
Suppose that $D$ is a \diag, and that $D'$ is a \diag{} obtained by applying one of \Reidemeister{} to $D$ once.
For any $(Q\times\Z_{k\type Q})_{\Wai}$-coloring $\si$ of $D$, it holds that $\operatorname{gcd}\si=\operatorname{gcd}\si'$,
where $\si'$ is the unique $(Q\times\Z_{k\type Q})_{\Wai}$-coloring of $D'$ described in Proposition~\ref{prop:coloring_invariant}.
\end{lemma}

For a \diag{} $D$, a shadow \mgr{} 2-cocycle $\theta:C_2(Q\times\Z_{k\type Q})_\Wai\to A$ and an integer $m$ with $1\le m \le {k\type Q}$,
we define the following multisets:
\begin{align*}
\mathcal{H}^m({D})&=\Set{[W(D;\si)]\in{H}_2(Q\times\Z_{k\type Q})_{\Wai}|\si\in\col_{Q\times\Z_{k\type Q}}(D)_{\Wai},\,\gcd c=m},\\
\Phi_\theta^m({D})&=\Set{\theta(W(D;\si))|\si\in\col_{Q\times\Z_{k\type Q}}(D)_{\Wai},\,\gcd c=m}.
\end{align*}
The following theorem follows from Theorem~\ref{theorem_oriented_surface_with_S1}, Proposition~\ref{prop:Y-orientation} and Lemma~\ref{lemma:gcd_new}.

\begin{theorem}\label{thm:main2}
Let $D$ be a \diag{} of an \surf{} $(F,o)$.
The multisets $\mathcal H^m(D)$ and $\Phi^m_\theta(D)$ are invariants of $(F,o)$, respectively.
\end{theorem}

Then we put $\mathcal{H}^m(F,o)=\mathcal{H}^m({D})$ and $\Phi^m_\theta(F,o)=\Phi^m_\theta({D})$, respectively.

\begin{proposition}\label{prop:minusinvers_localchainminus}
For an \surf{} $(F,o)$, it follows that
$\mathcal H(-F^*,-o^*)=-\mathcal H(F,o)$, $\Phi_\theta(-F^*,-o^*)=-\Phi_\theta(F,o)$, $\mathcal H^m(-F^*,-o^*)=-\mathcal H^m(-F^*,-o^*)$ and $\Phi^m_\theta(-F^*,-o^*)=-\Phi^m_\theta(F,o)$,
where we write $-S=\Set{-a|a\in S}$ for a multiset $S$.
\end{proposition}
\begin{proof}
Let $D$ (resp.~$-D^*$) be the \diag{} of $(F,o)$ (resp.~$(-F^*,-o^*)$), and let $c$ (resp.~$c^*$) be its  coloring defined in the proof of Proposition~\ref{prop:coloring_not_distinguish}.
It suffices to show that $W(D,c)=-W(-D^*,c^*)$.
We can observe that the local weight at each vertex or crossing in $D$ with $c$ has the opposite sign from the corresponding local weight in $-D^*$ with $c^*$ (see Figure~\ref{FIG:X_Y_coloring_proof_mirror}).
\end{proof}

\section{Constructions of shadow rack 2-cocycles}\label{SEC:2-COCYCLE_racks}
In this section, we suppose that ${Q}$ is a rack, ${\Wai}$ is a ${Q}$-set and $A$ is an abelian group.
\par
A 2-cochain $\theta:{C}_2^{\mathrm{R}}({Q})_{\Wai}\to{A}$ of ${C}^*_{\mathrm{R}}({Q};A)_{\Wai}$ is a cocycle if and only if the following condition is satisfied:
\begin{align*}
&\hspace{3ex}\theta\left(
	\la\upsilon\star{a}\ra\la{b}\ra\la{c}\ra
  +\la\upsilon\ra\la{a}\ra\la{c}\ra
  +\la\upsilon\star{c}\ra\la{a}*{c}\ra\la{b}*{c}\ra
	\right)\\
&=\theta\left(
   \la\upsilon\ra\la{b}\ra\la{c}\ra
  +\la\upsilon\star{b}\ra\la{a}*{b}\ra\la{c}\ra
  +\la\upsilon\ra\la{a}\ra\la{b}\ra
	\right),
\end{align*}
for any $\upsilon\in\Wai$ and $a,b,c\in{Q}$,
since
\begin{align*}
\partial_3(\la\upsilon\ra\la{a}\ra\la{b}\ra\la{c}\ra)
&=\la\upsilon\star{a}\ra\la{b}\ra\la{c}\ra
  +\la\upsilon\ra\la{a}\ra\la{c}\ra
  +\la\upsilon\star{c}\ra\la{a}*{c}\ra\la{b}*{c}\ra\\
&\phantom{=}
  -\la\upsilon\ra\la{b}\ra\la{c}\ra
  -\la\upsilon\star{b}\ra\la{a}*{b}\ra\la{c}\ra
  -\la\upsilon\ra\la{a}\ra\la{b}\ra.
\end{align*}
We call the above condition the \textit{(shadow) rack 2-cocycle condition}.

\begin{example}\label{ex_COCYCLES}{\rm
Let $f,g,h$ be elements of the polynomial ring $R[t^\pm,s]$ over a ring $R$.
Let $M$ be a left $R'$-module, where 
\[
R'=R[t^\pm,s]/\left(s(t+s-1), f(1-t),fs,(h(1-t)+(f+g)s), (f+g+h)s\right).
\]
For a $(t,s)$-rack $M$, we define a 2-cochain $\theta:{C}_2^{\mathrm{R}}(M)_{M}\to M$ by $\theta(x,y,z)=fx+gy+hz$.
Then $\theta$ is a 2-cocycle of ${C}^*_{\mathrm{R}}(M;M)_{M}$, since the shadow rack 2-cocycle condition is satisfied.
The details are left to the reader.
}\end{example}

Next proposition claims that we can obtain a shadow rack 2-cocycle $\oline{\theta}:{C}_2^{\mathrm{R}}(Q^n)_{\Wai}\to A$ from a given shadow rack 2-cocycle $\theta:{C}_2^{\mathrm{R}}({Q})_{\Wai}\to A$.

\begin{proposition}\label{prop:quandle rack}
Let $\theta:{C}_2^{\mathrm{R}}(Q)_{\Wai}\to A$ be a shadow rack 2-cocycle.
Fix ${e}_1,\ldots,{e}_m\in\{-1,1\}$ and ${i}_1,\ldots,{i}_m\in\{1,\ldots,n\}$ for arbitrary $m,n\in\Z_{>0}$.
\begin{enumerate}
\item
The Cartesian product ${Q}^n$ is a rack with the binary operation defined by
\begin{align*}
&({a}_1,\ldots,{a}_n)\ast({b}_1,\ldots,{b}_n)
=({a}_1{\ast}^{e_1}{b}_{i_1}{\ast}^{e_2}\cdots{\ast}^{e_m}{b}_{i_m},\ldots,{a}_n{\ast}^{e_1}{b}_{i_1}{\ast}^{e_2}\cdots{\ast}^{e_m}{b}_{i_m}). 
\end{align*}
Moreover, ${\Wai}$ is also a ${Q}^n$-set with the map $\star$ defined by
$$
\upsilon\star({a}_1,\ldots,{a}_n)=\upsilon{\star}^{e_1}{b}_{i_1}{\star}^{e_2}\cdots{\star}^{e_m}{b}_{i_m}.
$$
\item
Let $Q^n$ be the rack and $\Wai$ be the ${Q}^n$-set described in (1).
We define a 2-cochain $\oline{\theta}:{C}_2^{\mathrm{R}}(Q^n)_{\Wai}\to A$ by
\[
\oline{\theta}\left(
\la\upsilon\ra\la({a}_1,\ldots,{a}_n)\ra\la({b}_1,\ldots,{b}_n)\ra
\right)
=
\sum_{p=1}^m\sum_{q=1}^m
e_pe_q\,\theta\left(
\la\upsilon_{p,q}\ra
\la a_{p,q}\ra
\la b_{i_q}\ra
\right)
\]
for any generator $\la\upsilon\ra\la({a}_1,\ldots,{a}_n)\ra\la({b}_1,\ldots,{b}_n)\ra$ of ${C}_2^{\mathrm{R}}(Q^n)_\Wai$, where we set
\begin{align*}
\upsilon_{p,q}
&=\upsilon{\star}^{e_1}a_{i_1}{\star}^{e_2}\cdots{\star}^{e_{p-1}}a_{i_{p-1}}{\star}^{-\delta(e_p,-1)}a_{i_p}{\star}^{e_1}b_{i_1}{\star}^{e_2}\cdots{\star}^{e_{q-1}}b_{i_{q-1}}{\star}^{-\delta(e_q,-1)}b_{i_q},\\
a_{p,q}&=a_{i_p}{\ast}^{e_1}b_{i_1}{\ast}^{e_2}\cdots*^{e_{q-1}}b_{i_{q-1}}{\ast}^{-\delta(e_q,-1)}b_{i_q},
\end{align*}
and where $\delta$ is the Kronecker delta.
Then $\oline{\theta}$ is a shadow rack 2-cocycle.
\end{enumerate}
\end{proposition}
We note that, in general, the rack ${Q}^n$ in Proposition~\ref{prop:quandle rack} is not a quandle even if ${Q}$ is a quandle.

\begin{proof}
(1) This is proved in \cite[Proposition~3.1]{IMM20}.
\par
(2)
It is sufficient to verify that $\oline{\theta}$ vanishes on the chain 
$\partial_3\left(\la\upsilon\ra\la a\ra\la b\ra\la c\ra\right)$ for all $\upsilon\in\Wai$, $a,b,c\in Q^n$.
It is seen that $\oline{\theta}$ vanishes using the fact that $\theta$ is a shadow rack 2-cocycle.
This is observed by considering the geometric meaning of an $m$-parallel operation:
the local weight at a crossing $\chi$ of a \diag{} $D$ corresponds to the sum of the local weights at the corresponding $m^2$ crossings $\chi_{1,1},\ldots,\chi_{m,m}$ in the $m$-paralleled diagram $D^m$,
and $\oline{\theta}(w(\chi;c))\in A$ is the sum of $\theta(w(\chi_{1,1};c^m)),\ldots,\theta(w(\chi_{m,m};c^m))$,
where $c$ is a $(Q_n)_\Wai$-coloring of $D$ and $c^m$ is the corresponding $Q_\Wai$-coloring of $D^m$.
For example, we consider the case where $Q^2$ is the rack and $\Wai$ is the $Q^2$-set defined by
$({a}_1,a_2)\ast({b}_1,{b}_2)=(a_1\ast b_1{\ast}^{-1}b_2\ast b_1\ast b_2,\,a_2\ast b_1{\ast}^{-1}b_2\ast b_1\ast b_2)$ and $\upsilon\star(a_1,a_2)=\upsilon\star b_1{\star}^{-1}b_2\star b_1\star b_2$.
We note that $m=4$, $n=2$, $e_1=e_3=e_4=1, e_2=-1$, and $i_1=i_3=1, i_2=i_4=2$.
In Fig.~\ref{FIG:interpretation_product_rack_cos}, the geometric meaning of $\oline{\theta}$ is illustrated:
the local weight at a crossing $\chi_{p,q}$ in the 4-paralleled diagram is $e_pe_q\la\upsilon_{p,q}\ra\la a_{p,q}\ra\la b_{i_q}\ra$,
where $a_{p,q}$ (resp.~$b_{i_q}$) is a color of the corresponding underarc (resp.~overarc) and $\upsilon_{p,q}$ is a color of the corresponding region.
We have $a_{3,1}=a_1$, $a_{3,2}=a_{3,3}=a_1*b_1{*}^{-1}b_2$, $a_{3,4}=a_1*b_1{*}^{-1}b_2*b_1$ and
$\upsilon_{3,1}=\upsilon\star a_1{\star}^{-1}a_2$, $\upsilon_{3,2}=\upsilon_{3,3}=\upsilon\star a_1{\star}^{-1}a_2\star b_1{\star}^{-1}b_2$, $\upsilon_{3,4}=\upsilon\star a_1{\star}^{-1}a_2\star b_1{\star}^{-1}b_2\star b_1$ for instance.
\end{proof}
		\begin{figure}[htbp]\centerline{
		\includegraphics{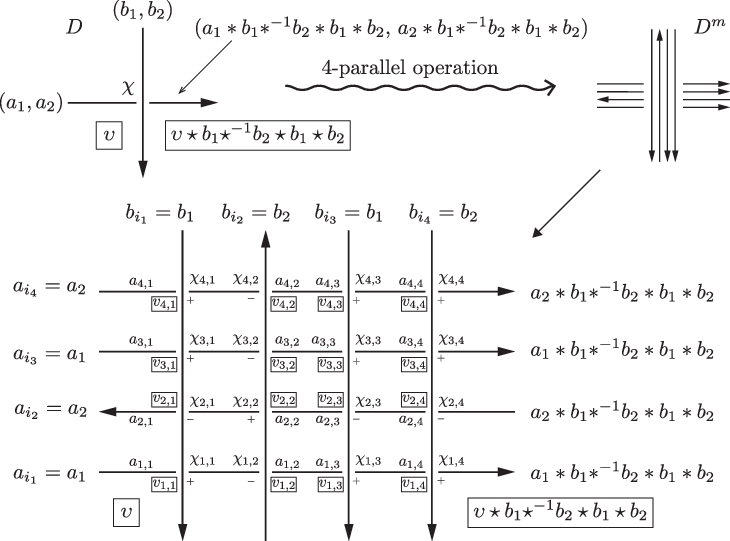}}
		\caption{A geometric interpretation of an $m$-parallel operation.}
		\label{FIG:interpretation_product_rack_cos}
		\end{figure}

\section{Constructions of shadow \mgr{} 2-cocycles}\label{SEC:2-COCYCLE_mgr}
In this section, we suppose that
${X}=\bigsqcup_{\lambda\in\Lambda}{G}_\lambda$ is an \mgr, ${\Wai}$ is an ${X}$-set and $A$ is an abelian group.
\par
A 2-cochain $\theta:{C}_2(X)_{\Wai}\to{A}$ of $C^*({X};A)_{{\Wai}}$ is a cocycle if and only if the following conditions (1), (2), (3) and (4) are satisfied.
\begin{enumerate}
\item For any ${\upsilon}\in\Wai$ and ${x},{y},{z}\in{X}$,
$$
 \theta\left(
\la{\upsilon}\star{x}\ra\la{y}\ra\la{z}\ra+\la{\upsilon}\ra\la{x}\ra\la{z}\ra+\la{\upsilon}\star{z}\ra\la {x}*{z}\ra\la{y}*{z}\ra
\right)
=\theta\left(
\la{\upsilon}\ra\la{y}\ra\la{z}\ra+\la{\upsilon}\star{y}\ra\la{x}*{y}\ra\la{z}\ra+\la{\upsilon}\ra\la{x}\ra\la{y}\ra
\right).
$$
\item For any ${\upsilon}\in\Wai$, ${x}\in{X}$ and ${y}_1,{y}_2\in{G_\lambda}$,
$$
\theta\left(
\la{\upsilon}\star{x}\ra\la{y}_1,{y}_2\ra+\la{\upsilon}\ra\la{x}\ra\la{y}_1{y}_2\ra
\right)
=\theta\left(
\la{\upsilon}\ra\la{y}_1,{y}_2\ra+\la{\upsilon}\star{y}_1\ra\la{x}*{y}_1\ra\la{y}_2\ra+\la{\upsilon}\ra\la{x}\ra\la{y}_1\ra
\right).
$$
\item For any ${\upsilon}\in\Wai$, ${x}_1,{x}_2\in G_\lambda$ and $y\in X$,
$$
\theta\left(
\la{\upsilon}\star{x}_1\ra\la{x}_2\ra\la{y}\ra+\la{\upsilon}\ra\la{x}_1\ra\la{y}\ra+\la{\upsilon}\star{y}\ra\la {x}_1*{y},{x}_2*{y}\ra
\right)
=\theta\left(
\la{\upsilon}\ra\la{x}_1{x}_2\ra\la{y}\ra+\la{\upsilon}\ra\la{x}_1,{x}_2\ra
\right).
$$
\item For any ${\upsilon}\in\Wai$ and ${x}_1,{x}_2,{x}_3\in{G_\lambda}$,
$$
\theta\left(
\la{\upsilon}\star{x}_1\ra\la{x}_2,{x}_3\ra+\la{\upsilon}\ra\la{x}_1,{x}_2{x}_3\ra
\right)
=\theta\left(
\la{\upsilon}\ra\la{x}_1{x}_2,{x}_3\ra+\la{\upsilon}\ra\la{x}_1,{x}_2\ra
\right).
$$
\end{enumerate}
These conditions come from Example~\ref{example_partial_3}.
We call them the \textit{(shadow) \mgr{} 2-cocycle conditions}.

\begin{proposition}
Let $X=\bigsqcup_{\lambda\in\Lambda}G_\lambda$ be an \mgr.
A 2-cochain $\theta:C_2(X)\to A$ is an \mgr{} 2-cocycle if and only if the set $X\times A=\bigsqcup_{\lambda\in\Lambda}(G_\lambda\times A)$ with
\[
(x,s)*(y,t)=\left(x*y,s+\theta(\la{x}\ra\la{y}\ra)\right),\quad
(x_1,s)(x_2,t)=\left(x_1x_2,s+t+\theta(\la{x_1,x_2}\ra)\right)
\]
is an \mgr.
\end{proposition}

\begin{proof}
Let $\lambda,\mu\in\Lambda$, $x,y,z\in X$, $x_1,x_2,x_3\in G_\lambda$, $y_1,y_2\in G_\mu$ and $s,t,u\in A$ be arbitrary elements, respectively.
\par
We show that the \mgr{} 2-cocycle condition~(4) is satisfied if and only if
$G_\lambda\times A$ is a group with the identity $(e_\lambda,-\theta(\lr{e_\lambda,e_\lambda}))$, and $(x_1,s)^{-1}:=(x_1^{-1},-s-\theta(\lr{e_\lambda,e_\lambda})-\theta(\lr{x_1,x_1^{-1}}))$ is the inverse element of $(x_1,s)\in G_\lambda\times A$.
The \mgr{} 2-cocycle condition~(4) is satisfied if and only if $G_\lambda\times A$ is associative, since
\begin{align*}
  \left((x_1,s)(x_2,t)\right)(x_3,u)
&=\left(x_1x_2x_3,s+t+\theta(\lr{x_1,x_2})+u+\theta(\lr{x_1x_2,x_3})\right),\\
  (x_1,s)\left((x_2,t)(x_3,u)\right)
&=\left(x_1x_2x_3,s+t+u+\theta(\lr{x_2,x_3})+\theta(\lr{x_1,x_2x_3})\right).
\end{align*}
Assume that the \mgr{} 2-cocycle condition~(4) is satisfied.
Then, we have $(x_1,s)(e_\lambda,-\theta(\lr{e_\lambda,e_\lambda}))=(x_1,s)=(e_\lambda,-\theta(\lr{e_\lambda,e_\lambda}))(x_1,s)$,
since $\theta(\lr{e_\lambda,x_1})=\theta(\lr{e_\lambda,e_\lambda})=\theta(\lr{x_1,e_\lambda})$.
We also have $(x_1,s)(x_1,s)^{-1}=\left(e_\lambda,-\theta(\lr{e_\lambda,e_\lambda})\right)=(x_1,s)^{-1}(x_1,s)$, since $\theta(\lr{x_1^{-1},x_1})=\theta(\lr{x_1,x_1^{-1}})$. 
Therefore, the claim holds.
\par
The axiom~(1) of \mgr s is satisfied if and only if the \mgr{} 2-cocycle condition~(2) is satisfied, since
\begin{align*}
  (x,s)*\left((y_1,t)(y_2,u)\right)
&=\left(x*(y_1y_2),s+\theta(\lr{x}\lr{y_1y_2})\right),\\
  ((x,s)*(y_1,t))*(y_2,u)
&=(\left(x*y_1)*y_2,s+\theta(\lr{x}\lr{y_1})+\theta(\lr{x*y_1}\lr{y_2})\right),\\
  (x,s)*\left(e_\mu,-\theta(\lr{e_\mu,e_\mu})\right)
&=\left(x,s+\theta(\lr{x}\lr{e_\mu})\right).
\end{align*}
Here, we note that $\theta(\lr{x}\lr{e_\mu})=0$ if the \mgr{} 2-cocycle condition (2) is satisfied.
\par
The axiom~(2) of \mgr s is satisfied if and only if the \mgr{} 2-cocycle condition~(1) is satisfied, since
\begin{align*}
  \left( (x,s)*(y,t)\right)*(z,u)
&=\left((x*y)*z,s+\theta(\lr{x}\lr{y})+\theta(\lr{x*y}\lr{z})\right),\\
  \left((x,s)*(z,u)\right)*((y,t)*(z,u))
&=((x*z)*(y*z),s+\theta(\lr{x}\lr{z})+\theta(\lr{x*z}\lr{y*z})).
\end{align*}

The axiom~(3) of \mgr s is satisfied if and only if the \mgr{} 2-cocycle condition~(3) is satisfied, since
\begin{align*}
  \left((x_1,s)(x_2,t)\right)*(y,u)
&=\left((x_1x_2)*y, s+t+\theta(\lr{x_1,x_2})+\theta(\lr{x_1x_2}\lr{y})\right),\\
  ((x_1,s)*(y,u))((x_2,t)*(y,u))
&=\left((x_1*y)(x_2*y),s+\theta(\lr{x_1}\lr{y})+t+\theta(\lr{x_2}\lr{y})+\theta(\lr{x_1*y,x_2*y})\right).
\end{align*}
This completes the proof.
\end{proof}

Next proposition claims that we can obtain a shadow \mgr{} 2-cocycle $\otilde \theta:{C}_2(Q\times\Z_{k\type{Q_\Wai}})_{{\Wai}}\to A$ from a given shadow rack 2-cocycle $\theta:{C}_2^{\mathrm{{R}}}({Q})_{{\Wai}}\to A$.

\begin{proposition}\label{prop_main_mgr_ccycle}
Let $\theta:{C}_2^{\mathrm{{R}}}({Q})_{{\Wai}}\to A$ be a shadow rack 2-cocycle.
For any $\upsilon\in\Wai$, $a,b\in Q$ and $i,j\in\Z\setminus\{0\}$, we set
$$\displaystyle\theta_{\upsilon;\,a,i;\,b,j}=\ccycl{|i|}{|j|}.$$
Let $X:=Q\times\Z_{k\type{Q_\Wai}}$ be the associated \mgr{} and $Y$ the $X$-set described in Remark~\ref{frequentlyused_mgr}.
\begin{enumerate}
\item Suppose that $k\type{Q_\Wai}=\infty$, that is, $X=Q\times\Z$.
We define a 2-cochain $\otilde \theta:{C}_2(X)_{{\Wai}}\to A$ as follows:
for any generators $\la\upsilon\ra\la({a},i)\ra\la({b},j)\ra$ and $\la\upsilon\ra\la(c,i),(c,j)\ra$ of ${C}_2\left(X\right)_{\Wai}$,
\begin{align*}
\otilde\theta\left(\la\upsilon\ra\la({a},i)\ra\la({b},j)\ra\right)&= \left\{
\begin{array}{ll}
\phantom{-}\theta_{\upsilon;\,a,i;\,b,j} & (i>0,\,j>0),\\
-          \theta_{\upsilon\star^{j}b;\,a*^{j}b,i;\,b,-j}& (i>0,\,j<0),\\
-          \theta_{\upsilon\star^{i}a;\,a,-i;\,b,j}& (i<0,\,j>0),\\
\phantom{-}\theta_{\upsilon\star^{i}a\star^{j}b;\,a*^{j}b,-i;\,b,-j}& (i<0,\,j<0),\\
\phantom{-}0& (otherwise),
\end{array}
\right.\\
\otilde\theta\left(\la\upsilon\ra\la({c},i),({c},j)\ra\right)&=0.
\end{align*}
Then $\otilde{\theta}$ is a shadow \mgr{} 2-cocycle.
\item Suppose that $k\type{Q_\Wai}\ne\infty$.
We define a 2-cochain $\otilde \theta:{C}_2(X)_{{\Wai}}\to A$ as follows:
for any generators $\la\upsilon\ra\la({a},i)\ra\la({b},j)\ra$ and $\la\upsilon\ra\la(c,i),(c,j)\ra$ of ${C}_2\left(X\right)_{\Wai}$,
\begin{align*}
\otilde\theta\left(\la\upsilon\ra\la({a},i)\ra\la({b},j)\ra\right)&=\theta_{\upsilon;a,i;b,j},\\
\otilde\theta\left(\la\upsilon\ra\la({c},i),({c},j)\ra\right)&=0,
\end{align*}
where $1\le i\le k\type{Q_\Wai}$ and $1\le j\le k\type{Q_\Wai}$.
\begin{enumerate}
\item[(i)]
$\otilde{\theta}$ is a shadow \mgr{} 2-cocycle if the following equations hold for any $\upsilon\in\Wai$ and $a,b\in{Q}$:
\begin{align*}
\sum_{p=1}^{k\type Q_\Wai}
\theta\left(\la\upsilon\star^{p-1}a\ra\la a\ra\la b\ra\right)&=0,\\
\sum_{q=1}^{k\type Q_\Wai}
\theta\left(\la\upsilon\star^{q-1}b\ra\la a*^{q-1}b\ra\la b\ra\right)&=0.
\end{align*}
\item[(ii)]
$\otilde{\theta}$ is a shadow \mgr{} 2-cocycle if $k\alpha=0$ for any $\alpha\in A$.
In particular, $\otilde{\theta}$ is a shadow \mgr{} 2-cocycle if $A$ is finite and $k=\#A$.
\end{enumerate}
\end{enumerate}
\end{proposition}

\begin{proof}
(1)
It is sufficient to verify that $\otilde{\theta}$ vanishes on the following chains:
\begin{align}
&\partial_3(\la\upsilon\ra\la x\ra\la y\ra\la z\ra)\label{eq:rackcos},\\
&\partial_3(\la x\ra\la y_1,y_2\ra),\quad\partial_3(\la x_1,x_2\ra\la y\ra)\label{eq:mgrcos1},\\
&\partial_3(\la x_1,x_2,x_3\ra).\label{eq:mgrcos2}
\end{align}
On chains in \eqref{eq:rackcos}, it is seen that $\otilde{\theta}$ vanishes using the fact that $\theta$ is a shadow rack 2-cocycle.
On chains in \eqref{eq:mgrcos1}, we see that $\otilde{\theta}$ vanishes without using any assumption on $\theta$.
These are observed by considering the geometric meaning of a parallel operation according to the second coordinate $i$ of an element of $Q\times\Z$ as illustrated in Fig.~\ref{FIG_mgr_parallel}:
we take $-i$ edges with opposite directions if $i$ is negative.
The local weight at a crossing $\chi$ of a \diag{} $D$ corresponds to the sum of the local weights at the corresponding $ij$ crossings $\chi_{1,1},\ldots,\chi_{i,j}$ in the paralleled diagram $D'$,
and $\otilde{\theta}(w(\chi;c))\in A$ is the sum of $\theta(w(\chi_{1,1};c')),\ldots,\theta(w(\chi_{i,j};c'))$,
where $c$ is an $X_\Wai$-coloring of $D$ and $c'$ is the corresponding $Q_\Wai$-coloring of $D'$.
In Fig.~\ref{FIG_mgr_parallel}, 
the local weight at a crossing $\chi_{p,q}$ in the paralleled diagram is $\la\upsilon\star^{p-1}a\star^{q-1}b\ra\la a*^{q-1}b\ra\la b\ra$,
where $a*^{q-1}b$, $b$ and $\upsilon\star^{p-1}a\star^{q-1}b$ are the colors of the corresponding underarc, overarc and region, respectively.
By the assumption $\otilde{\theta}\left(\la\upsilon\ra\la({c},i),({c},j)\ra\right)=0$, we see that $\otilde{\theta}$ vanishes on chains in \eqref{eq:mgrcos2}.
\par
(2)
(i)
We denote by $\otilde\theta_\infty$ the shadow \mgr{} 2-cocycle $\otilde\theta$ defined in (1) of Proposition~\ref{prop_main_mgr_ccycle}.
Let $i$ and $j$ be positive integers with $i,j\le \kappa$, where we put $\kappa=k\type{Q_\Wai}$.
It is sufficient to show that $\otilde\theta_\infty\left(\la\upsilon\ra\la({a},i+s\kappa)\ra\la({b},j+t\kappa)\ra\right)=\otilde\theta\left(\la\upsilon\ra\la({a},i)\ra\la({b},j)\ra\right)$ for any $\upsilon\in\Wai$, $a,b\in Q$ and $s,t\in\Z$.
We note that $\theta_{\upsilon;\,a,\kappa;\,b,1}=0$ and $\theta_{\upsilon;\,a,1;\,b,\kappa}=0$ by the assumption.
In the case that $i+s\kappa>0$ and $j+t\kappa>0$, we have
\begin{align*}
\otilde\theta_\infty\left(\la\upsilon\ra\la({a},i+s\kappa)\ra\la({b},j+t\kappa)\ra\right)
&=\theta_{\upsilon;\,a,i+s\kappa;\,b,j+t\kappa}\\
&=\sum_{p=1}^{i+s\kappa}\theta_{{\upsilon}\star^{p-1}a;\,a,1;\,b,j+t\kappa}\\
&=\sum_{p=1}^{i}\theta_{{\upsilon}\star^{p-1}a;\,a,1;\,b,j}\\
&=\theta_{{\upsilon};\,a,i;\,b,j}\\
&=\otilde\theta\left(\la\upsilon\ra\la({a},i)\ra\la({b},j)\ra\right),
\end{align*}
where, the third equation follows from the assumption.
In the other cases, we have $\otilde\theta_\infty\left(\la\upsilon\ra\la({a},i+s\kappa)\ra\la({b},j+t\kappa)\ra\right)=\otilde\theta\left(\la\upsilon\ra\la({a},i)\ra\la({b},j)\ra\right)$ in the same manner.
\par
(ii)
By using (i), it suffices to prove that
$\theta_{\upsilon;a,\kappa;b,1}=\theta_{\upsilon;a,1;b,\kappa}=0$.
We have
\begin{align*}
\theta_{\upsilon;a,\kappa;b,1}
=\sum_{{p}=1}^{\kappa}\theta\left(\la\upsilon\star^{{p}-1}a\ra\la{a}\ra\la{b}\ra\right)
=k\cdot\theta\left(\sum_{{p}=1}^{\type{Q_\Wai}}\la\upsilon\star^{{p}-1}a\ra\la{a}\ra\la{b}\ra\right)
=0.
\end{align*}
Similarly, we have
$\theta_{\upsilon;a,1;b,\kappa}=0$.
\end{proof}

		\begin{figure}[htbp]\centerline{
		\includegraphics{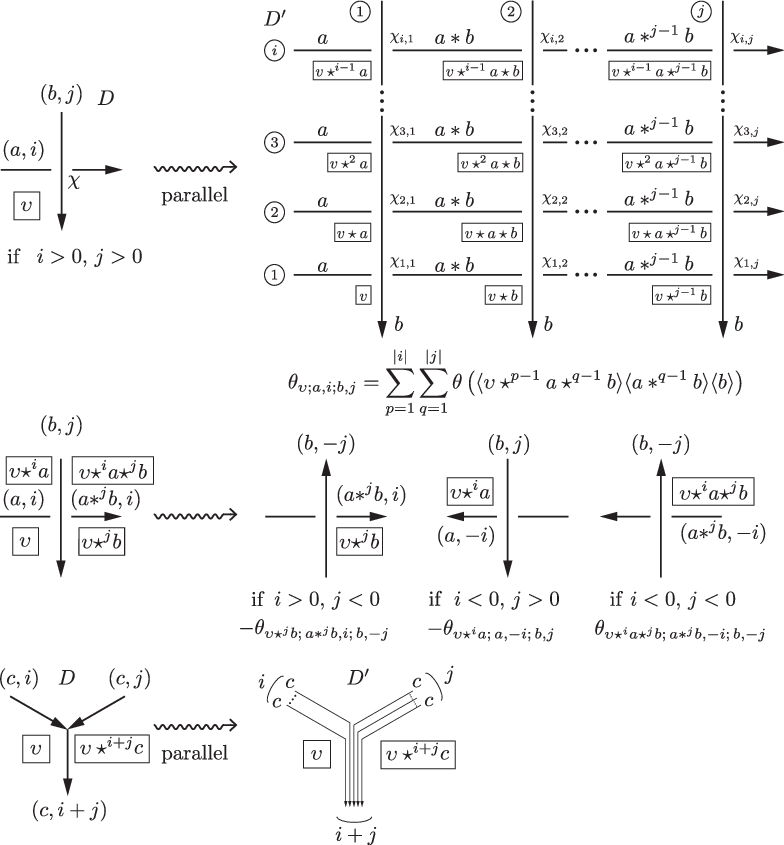}}
		\caption{A geometric interpretation of a parallel operation.}
		\label{FIG_mgr_parallel}
		\end{figure}

\section{Symmetries of \oss s}\label{SEC:SYMMETRIES}
In this section, we assume that an \oss{} may have closed components.
If an \oss{} $F$ is connected and closed, we have the \oss{} $F^\circ$ with a boundary by removing an open disk from $F$.
Two closed and connected \oss s $F$ and $F'$ are equivalent if and only if $F^\circ$ and $F'^\circ$ are equivalent (see~\cite{Matsuzaki19}). 
Then, we often identify $F$ with $F^\circ$ and regard a \diag{} of $F^\circ$ as that of $F$.
\par
Let define some terminology on symmetries of an \oss{} $F$.
If $F\cong -F$ then $F$ is \textit{invertible}.
By considering combinations of invertibility and chirality, we can give the following  five definitions of symmetries\footnote{If $F$ satisfies (1), (3) or (4), $F$ is called \textit{amphicheiral}.
If $F$ satisfies (2) or (5), $F$ is called \textit{chiral}.}.
\begin{enumerate}
\item If $-F\cong F\cong F^*$, $F$ is \textit{fully amphicheiral}.
\item If $-F\cong F\not\cong F^*$, $F$ is \textit{reversible}.
\item If $-F\not\cong F\cong F^*$, $F$ is \textit{positive amphicheiral}.
\item If $-F\not\cong F\not\cong F^*$ and $F\cong-F^*$, $F$ is \textit{negative amphicheiral}.
\item If $-F\not\cong F\not\cong F^*$ and $F\not\cong-F^*$, $F$ is \textit{fully chiral}.
\end{enumerate}
We denoted by $O_g$ the connected and closed \oss{} with genus $g\ge0$ standardly embedded in $S^3$, that is, it bounds two handlebodies in $S^3$.
Clearly, $O_g$ is fully amphicheiral.
\par
Evaluating our invariant (by using Mathematica), we give an example of positive amphicheiral \oss s  that have a boundary (Example.~\ref{positive_amphicheiral}),
and also give examples of these five symmetric types of closed \oss s respectively for any genus $g\ge2$ (Proposition~\ref{proposition:table}).

\begin{example}\label{example:prime}
Let $3_1$ (resp.~$4_1$) be the \oss{} obtained from the boundary of a regular neighborhood of the right-handed trefoil (resp.~figure-eight knot) such that its front side is facing $\infty\in S^3$ as depicted in Fig.~\ref{FIG:prime_surfaces}.
Since the right-handed trefoil (resp.~figure-eight knot) is chiral (resp.~amphicheiral), then $3_1$ (resp.~$4_1$) is fully chiral (resp.~negative amphicheiral).
Here, we note that $-F\not\cong F\not\cong F^*$
if $F$ is not standardly embedded and has genus 1,
since one of the connected components of its exterior is the interior of a solid torus and the other is not.
Let $2_1^2$ be the oriented annulus so that it is an Seifert surface of the oriented Hopf link as illustrated in Fig.~\ref{FIG:prime_surfaces}.
Since any oriented annulus is always invertible, and the linking numbers of the boundaries of $2_1^2$ and $(2_1^2)^*$ are different, then $2_1^2$ is reversible.
\end{example}
		\begin{figure}[htbp]\centerline{
		\includegraphics{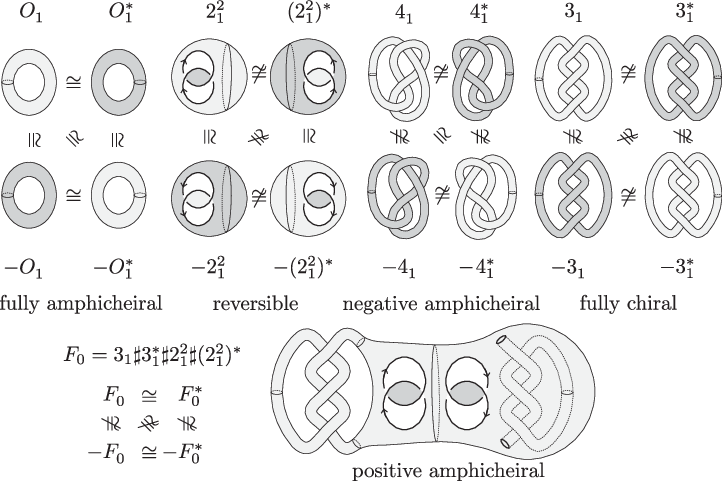}}
		\caption{Oriented spatial surfaces $O_1$, $3_1$, $4_1$ and $2_1^2$ etc.}
		\label{FIG:prime_surfaces}
		\end{figure}

Suppose that $S^3$ has a fixed orientation.
Let $(F_i\subset S^{3})$ be a pair of a connected \oss{} and the 3-sphere ($i=1,2$).
Let $B_i\subset S^{3}$ be a 3-ball such that
$B_i\cap \partial F_i=\emptyset$, and
$\partial B_i$ transversally intersects with $F_i$,
and $B_i\cap F_i$ is a disk ($i=1,2$).
The \textit{connected sum} $(F_1\subset S^{3})\sharp(F_2\subset S^{3})$, simply $F_1\sharp F_2$, is a new pair of a connected \oss{} and the $3$-sphere obtained by matching the boundaries
$\partial(S^{3}\setminus\operatorname{Int}B_1)$ and $\partial(S^{3}\setminus \operatorname{Int}B_2)$ using an orientation-reversing homeomorphism $r$ such that $r(\partial(F_1\setminus \operatorname{Int}B_1))=\partial(F_2\setminus \operatorname{Int}B_2)$, and that $r|_{F_1\setminus \operatorname{Int}B_1}$ is also orientation-reversing.
Up to equivalence, the operation $\sharp$ is well-defined, associative and commutative by the Alexander's theorem \cite{A1924} and the homogeneity theorem of Newman-Gugenheim \cite{G1953}.

\begin{example}\label{positive_amphicheiral}
We show that the \oss{} $F_0=3_1\sharp3_1^*\sharp 2_1^2\sharp(2_1^2)^*$ is positive amphicheiral (see Fig.~\ref{FIG:prime_surfaces}).
Since $F_0^*\cong3_1^*\sharp 3_1\sharp (2_1^2)^*\sharp2_1^2\cong F_0$,
it suffices to confirm $F_0\not\cong-F_0^*$.
Unfortunately, none of coloring numbers distinguish $F_0$ and $-F_0^*$ (see Proposition~\ref{prop:coloring_not_distinguish}).
We then prepare a shadow rack 2-cocycle.
Let $R_p$ be the dihedral quandle for a prime number $p$.
The following 2-cochain $\theta_p:C_2^\mathrm{R}(R_p)_{R_p}\to\Z_p$ is a shadow rack 2-cocycle of $C_\mathrm{R}^*(R_p;\Z_p)_{R_p}$;
\[
 \theta_p\left(\la\ai\ra\la\jei\ra\la\kei\ra\right)
=\left[(\ai-\jei)\frac{\,(\jei^p+(2\kei-\jei)^p-2\kei^p)\,}{p}\right],
\]
for any generator $\la\ai\ra\la\jei\ra\la\kei\ra\in C^{\mathrm R}_2(R_p)_{R_p}$, which is known as the Mochizuki's 3-cocycle (see Remark~\ref{REM:iikae_ccycle}).
Let $Q=(R_3)^3$ be the rack, described in Proposition~\ref{prop:quandle rack},
defined by $$(a_1,a_2,a_3)\ast(b_1,b_2,b_3)=(a_1\ast b_1,\,a_2\ast b_1,\,a_3\ast b_1),$$
and let $R_3$ be the $Q$-set defined by $\upsilon\ast(b_1,b_2,b_3)=\upsilon\ast b_1$, where $\type Q_{R_3}=2$.
By Proposition~\ref{prop:quandle rack}, we have the shadow rack 2-cocycle $\theta:=\oline{\theta_3}$ of $C^*_\mathrm{R}(Q;\Z_3)_{R_3}$.
Moreover, we obtain the shadow \mgr{} 2-cocycle $\otilde{\theta}$ of $C^*(Q\times\Z_2;\Z_3)_{R_3}$
since it satisfies the equations in Proposition~\ref{prop_main_mgr_ccycle} (2)(i).
We obtain $\Phi_{\otilde{\theta}}(F_0)=\left\{0_{\ul{2880}},1_{\ul{1728}},2_{\ul{1152}}\right\}$,
where we represent the multiplicity of elements of a multiset by using subscripts with underlines.
Then $\Phi_{\otilde{\theta}}(-F_0^*)=-\Phi_{\otilde{\theta}}(F_0)=\left\{0_{\ul{2880}},1_{\ul{1152}},2_{\ul{1728}}\right\}$ by Proposition~\ref{prop:minusinvers_localchainminus}.
Therefore, it holds that $F_0\not\cong-F_0^*$.
\end{example}
For Proposition~\ref{proposition:table}, we prepare the following lemma.
Set
$F_1=4_1\sharp(-4_1)$,
$F_2=3_1\sharp(-3_1)$,
$F_3=3_1\sharp3_1^* $,
$F_4=4_1$ and
$F_5=3_1$ as illustrated in Fig.~\ref{FIG_TABLE}.
		\begin{figure}[htbp]\centerline{
		\includegraphics{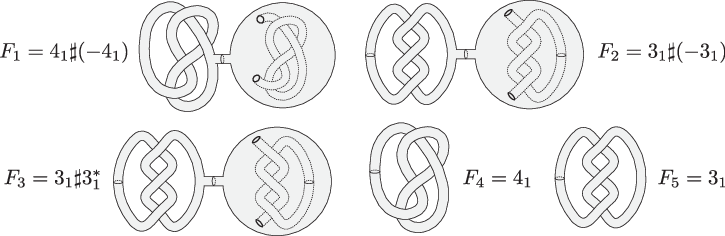}}
		\caption{Five closed \oss s.}
		\label{FIG_TABLE}
		\end{figure}
		
\begin{lemma}\label{essential_lemma}
For any $g\ge0$, it holds that $F_i\sharp O_g\not\cong-(F_i\sharp O_g)^*$ $(i=2,3,5)$.
\end{lemma}

\begin{proof}
Let  $Q=(R_3)^3$ be the rack defined by $$(a_1,a_2,a_3)\ast(b_1,b_2,b_3)=(a_1\ast b_1\ast b_2\ast b_3,\,a_2\ast b_1\ast b_2\ast b_3,\,a_3\ast b_1\ast b_2\ast b_3),$$
and let $R_3$ be the $Q$-set defined by $\upsilon\ast(b_1,b_2,b_3)=\upsilon\ast b_1\ast b_2\ast b_3$, where $\type Q_{R_3}=2$.
Set $Q_1=\left\{(0,0,0),(1,1,1),(2,2,2)\right\}\subset Q$ and $Q_2=Q\setminus Q_1$, then
\begin{equation}\label{eq:type1_is_3}
\text{$\type q_i=i$ for any $q_i\in Q_i$ ($i=1,2$).}
\end{equation}
Using Proposition~\ref{prop:quandle rack}, we have the shadow rack 2-cocycle $\theta:=\oline{\theta_3}$ of $C^*_\mathrm{R}(Q;\Z_3)_{R_3}$,
where $\theta_3$ is the Mochizuki's 3-cocycle.
Moreover, we obtain the shadow \mgr{} 2-cocycle $\otilde{\theta}$ of $C^*(Q\times\Z_6;\Z_3)_{R_3}$
by using Proposition~\ref{prop_main_mgr_ccycle}.
Suppose that $D$ is a \diag{} of a connected \oss{} $F$.
For $\alpha\in\mathcal A(D)$ and $q\in Q$, we put
\begin{align*}
\col^{3,6}(D)_\alpha^{q}&=\Set{c\in\col_{Q\times\Z_6}(D)_{R_3}|\gcd c\in\{3,6\},\,\operatorname{pr}_1\mathbin{\circ} c(\alpha)=q},\\
\Phi^{{3,6}}(F)_\alpha^{q}
&=\Set{\otilde\theta(W(D;c))\in\Phi_{\otilde\theta}(F)|c\in\col^{3,6}(D)_\alpha^q},
\end{align*}
where $\operatorname{pr}_1:Q\times \Z_{k\type Q}\to Q$ is the canonical projection.
We write
\[
\text{
$S\sqcup S'=
\left\{
{s_1}_{\ul{i_1}},\ldots,
{s_k}_{\ul{i_k+j_k}},\ldots,
{s_{\ell}}_{\ul{i_\ell+j_\ell}},\ldots,
{s_{m}}_{\ul{j_{m}}}
\right\}
$
and
$\bigsqcup_{i=1}^nS=\underbrace{S\sqcup\cdots\sqcup S}_{n}$
}
\]
for multisets
$S=\left\{{s_1}_{\ul{i_1}},\ldots,{s_k}_{\ul{i_k}},\ldots,{s_\ell}_{\ul{i_\ell}}\right\}$ and
$S'=\left\{{s_k}_{\ul{j_k}},\ldots,{s_\ell}_{\ul{j_\ell}},\ldots,{s_{m}}_{\ul{j_{m}}}\right\}$.
Here, $\sum_{q\in Q}\#\col^{3,6}(D)_\alpha^{q}$ and $\bigsqcup_{q\in Q}\Phi^{{3,6}}(F)_\alpha^{q}$ are invariants of $F$ (see~Theorem~\ref{thm:main2}).
\par
We show the following claim:
for any $\alpha\in\mathcal A(D)$,
\begin{equation}\label{equation:claim}
\Phi^{{3,6}}(F\sharp O_1)
=\left(\bigsqcup_{q_1\in Q_1}\bigsqcup_{i=1}^4\Phi^{{3,6}}(F)_\alpha^{q_1}\right)
\sqcup
 \left(\bigsqcup_{q_2\in Q_2}\bigsqcup_{i=1}^2\Phi^{{3,6}}(F)_\alpha^{q_2}\right).
\end{equation}
Let $D\sharp O_1$ be a \diag{} of $F\sharp O_1$ so that $D\sharp O_1$ divides into the $D$ part and a torus part, and that the torus part is connected to $\alpha$ as illustrated in Fig.~\ref{FIG_TABLE_DIAGRAM_MIMI}.
Let $c_\alpha^q\in\col^{3,6}(D)_\alpha^q$ and ${c_1}_\alpha^q\in\col^{3,6}(D\sharp O_1)_\alpha^q$ so that $D$ part of ${c_1}_\alpha^q$ is the same as $c_\alpha^q$.
The coloring condition of $D\sharp O_1$ is obtained from that of $D$ by adding the equality $q*^bq=q$ at the crossing $\chi$ in Fig.~\ref{FIG_TABLE_DIAGRAM_MIMI}.
The equality is satisfied if and only if $b\in\Set{k\in\{3,6\}|\text{$\type q$ divides $k$}}$.
Then $\#\col^{3,6}(D\sharp O_1)_\alpha^q=2\cdot\#\Set{k\in\{3,6\}|\text{$\type q$ divides $k$}}\cdot\#\col^{3,6}(D)_\alpha^q$.
By \eqref{eq:type1_is_3}, we obtain
\[
\#\col^{{3,6}}(D\sharp O_1)_\alpha^{q}= 
\begin{cases}
4\cdot\#\col^{{3,6}}(D)_\alpha^{q}&(q\in Q_1),\\
2\cdot\#\col^{{3,6}}(D)_\alpha^{q}&(q\in Q_2).
\end{cases}
\]
Since $\otilde\theta_2(\lr{\upsilon}\lr{(q,a)}\lr{(q,b)})=0$ when $q*^bq=q$,
we have $\otilde\theta_2(W(D\sharp O_1;{c_1}_\alpha^q))=\otilde\theta_2(W(D;{c}_\alpha^q))$, and the claim \eqref{equation:claim} follows.
\par
Let $D_i$ and $\alpha$ be the \diag{} of $F_i$ and the arc as illustrated in Fig.~\ref{FIG_TABLE_DIAGRAM} ($i=2,3,5$).
We obtain
\[
\Phi^{{3,6}}(F_2)_\alpha^{q}= 
\begin{cases}
\left\{0_{\ul{48}},1_{\ul{48}},2_{\ul{96}}\right\}&(q\in Q_1),\\
\left\{0_{\ul{12}},\phantom{,1_{48}}2_{\ul{96}}\right \}       &(q\in Q_2).
\end{cases}
\]
By \eqref{eq:type1_is_3} and \eqref{equation:claim}, we have inductively
\begin{align*}
\Phi^{3,6}(F_2\sharp O_g)
&=\bigsqcup_{i=1}^{3\cdot4^g}\left\{0_{\ul{48}},1_{\ul{48}},2_{\ul{96}}\right\}\sqcup\bigsqcup_{i=1}^{24\cdot2^g}\left\{0_{\ul{12}},2_{\ul{96}}\right\} \\
&=\left\{0_{\ul{144(4^g+2^{g+1})}},1_{\ul{144\cdot 4^g}},2_{\ul{288(4^g+2^g)}}\right\}.
\end{align*}
Then $\Phi^{3,6}(-(F_2\sharp O_g)^*)=-\Phi^{3,6}(F_2\sharp O_g)=\left\{0_{\ul{144(4^g+2^{g+1})}},1_{\ul{288(4^g+2^g)}},2_{\ul{144\cdot 4^g}}\right\}$ by Proposition~\ref{prop:minusinvers_localchainminus}.
Since $144\cdot 4^g\ne 288(4^g+2^g)$ for any $g$,
it holds that $F_2\sharp O_g\not\cong-(F_2\sharp O_g)^*$.
\par
In the same manner, since
\[
\Phi^{{3,6}}(F_3)_\alpha^{q}= 
\begin{cases}
\left\{0_{\ul{96}},1_{\ul{48}},2_{\ul{48}}\right\}&(q\in Q_1),\\
\left\{0_{\ul{12}},1_{\ul{12}}\phantom{,2_{48}}\right\}       &(q\in Q_2),
\end{cases}
\quad
\Phi^{{3,6}}(F_5)_\alpha^{q}= 
\begin{cases}
\left\{0_{\ul{48}},1_{\ul{48}},2_{\ul{96}}\right\}            &(q\in Q_1),\\
\left\{0_{\ul{12}},\phantom{1_{\ul{48}},}\,2_{\ul{96}}\right\}&(q\in Q_2),
\end{cases}
\]
then we have
$\Phi^{3,6}(F_3\sharp O_g)
= \left\{0_{\ul{288(4^g+2^g)}},
     1_{\ul{144(4^g+2^{g+1})}},
2_{\ul{144\cdot 4^g}}\right\}$ and
$\Phi^{3,6}(F_5\sharp O_g)
=\left\{0_{\ul{144(4^g+2^{g+1})}},
2_{\ul{144\cdot 4^g}}\right\}$.
Hence, $F_i\sharp O_g\not\cong-(F_i\sharp O_g)^*$ ($i=3,5$).
\end{proof}
		\begin{figure}[htbp]\centerline{
		\includegraphics{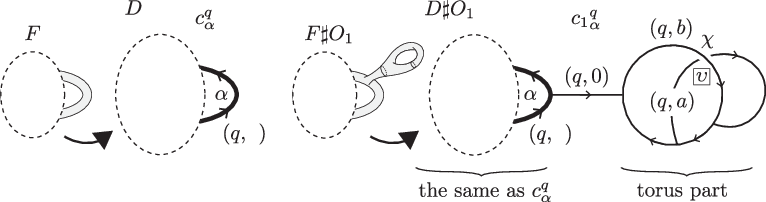}}
		\caption{Diagrams $D$ and $D\sharp O_1$ of $F$ and $F\sharp O_1$, and their colorings $c_\alpha^q$ and ${c_1}_\alpha^q$ by the \mgr{} $Q\times\Z_6$, respectively.}
		\label{FIG_TABLE_DIAGRAM_MIMI}
		\end{figure}

		\begin{figure}[htbp]\centerline{
		\includegraphics{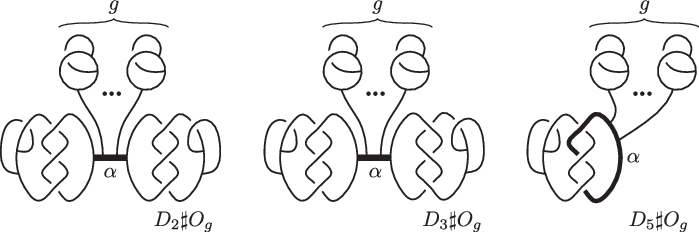}}
		\caption{A diagram $D_i\sharp O_g$ of $F_i\sharp O_g$ ($i=2,3,5$),
		whose Y-orientation is omitted for simplicity.}
		\label{FIG_TABLE_DIAGRAM}
		\end{figure}

\begin{proposition}\label{proposition:table}
For any $g\ge2$, there are genus $g$ closed \oss s of the five symmetric types, respectively.
Actually, we have the following table.
\begin{tabular}{r|c|c|c|c|c}\hline
 {\small genus}
&{\small fully amphicheiral}
&{\small reversible}
&{\small positive amphicheiral}
&{\small negative amphicheiral}
&{\small fully chiral}\\
\hline\hline
 {\small 1}
&$O_1$ only
&none
&none
&$F_4$
&$F_5$\\
\hline
 {\small $g\ge2$}
&$F_1\sharp O_{g-2}$
&$F_2\sharp O_{g-2}$
&$F_3\sharp O_{g-2}$
&$F_4\sharp O_{g-1}$
&$F_5\sharp O_{g-1}$\\
\hline
\end{tabular}
\end{proposition}

\begin{proof}
Any \oss{} $F$ with genus $g\ge2$ in the table is not standardly embedded, since $F$ bounds a 3-manifold which is not a handlebody. (For example, it follows from the Seifert-van Kampen theorem and the Grushko decomposition theorem.)
\par
As seen in Example~\ref{example:prime}, $F_4$ (resp.~$F_5$) is negative amphicheiral (resp.~fully chiral). 
If a closed \oss{} $F$ with genus 1 is not standardly embedded, then $-F\not\cong F\not\cong F^*$, that is, $F$ is not fully amphicheiral, not reversible and not positive amphicheiral.
Then the only $O_1$ is fully amphicheiral.
\par
Suppose that $g\ge0$, hereafter.
Since 
$-(F_1   \sharp O_{g})\cong(-4_1)\sharp4_1\sharp O_{g}\cong F_1\sharp O_{g}$ and
$(F_1\sharp O_{g})^*\cong4_1^* \sharp(-4_1)^*\sharp O_{g}\cong (-4_1) \sharp4_1\sharp O_{g}\cong F_1\sharp O_{g}$ (see~Fig.~\ref{FIG:prime_surfaces}),
then $F_1\sharp O_{g}$ is fully amphicheiral.
Since $-(F_2\sharp O_g)\cong (-3_1)\sharp 3_1 \sharp O_g\cong F_2\sharp O_g$ and $F_2\sharp O_g\not\cong-(F_2\sharp O_g)^*$ by Lemma \ref{essential_lemma},
then $F_2\sharp O_g$ is reversible.
Since $(F_3\sharp O_g)^*\cong 3_1^*\sharp 3_1\sharp O_g\cong F_3\sharp O_g$ and $F_3\sharp O_g\not\cong-(F_3\sharp O_g)^*$ by Lemma \ref{essential_lemma},
then $F_3\sharp O_g$ is positive amphicheiral.
The negative amphicheirality of $F_4\sharp O_g$ follows from the negative amphicheirality of $4_1$.
The fully chirality of $F_5\sharp O_g$ follows from the fully chirality of $3_1$ and Lemma~\ref{essential_lemma}.
Here, we note that $-(F_i\sharp O_g)\not\cong F_i\sharp O_g\not\cong (F_i\sharp O_g)^*$ ($i=4,5$) because $F_i\sharp O_g$ bounds both a handlebody and a 3-manifold which is not a handlebody in $S^3$.
\end{proof}

\section*{Acknowledgement}
The second author was supported by JSPS KAKENHI Grant Numbers JP20K22312 and JP21K13796.


\begin{thebibliography}{000}

\bibitem{A1924}
J.~W.~Alexander,
\emph{On the subdivision of $3$-sphere by a polyhedron},
Proc.\ Nat.\ Acad.\ Sci.\ U.~S.~A.\ \textbf{10} (1924),
6--8.

\bibitem{AS2005}
S.~Asami and
S.~Satoh,
\emph{An infinite family of non-invertible surfaces in 4-space},
Bull.\ London Math.\ Soc.\ (2)\text{37} (2005),
285--296.

\bibitem{BZ1967}
G.~Burde and
H.~Zeischang
\emph{Neuwirthsche Knoten und Fl\"{a}chenabbildungen},
Abh.\ Math.\ Sem.\ Univ.\ Hamburg \textbf{31} (1967),
239--246.

\bibitem{CES2004}
J.~S.~Carter,
M.~Elhamdadi and
M.~Saito,
\emph{Homology theory for the set-theoretic Yang-Baxter equation and knot invariants from generalizations of quandles},
Fund.\ Math.\ 184 (2004),
3--54.

\bibitem{CIST2017}
J.~S.~Carter,
A.~Ishii,
M.~Saito and
K.~Tanaka,
\emph{Homology for quandles with partial group operations},
Pacific J.\ Math.\ 287-1 (2017),
19--48.

\bibitem{CJKLS1999}
J.~S.~Carter,
D.~Jelsovsky,
S.~Kamada,
L.~Langford and
M.~Saito,
\emph{State-sum invariants of knotted curves and surfaces from quandle cohomology},
Electron.\ Res.\ Announc.\ Amer.\ Math.\ Soc.\ \textbf{5} (1999),
146--156.

\bibitem{CJKLS2003}
J.~S.~Carter,
D.~Jelsovsky,
S.~Kamada,
L.~Langford and
M.~Saito,
\emph{Quandle cohomology and state-sum invariants of knotted curves and surfaces},
Trans.\ Amer.\ Math.\ Soc.\ \textbf{355} (2003),
3947--3989.

\bibitem{CJKS2001}
J.~S.~Carter,
D.~Jelsovsky,
S.~Kamada and
M.~Saito,
\emph{Quandle homology groups, their Betti numbers, and virtual knots},
J.\ Pure Appl.\ Algebra \textbf{157} (2001),
135--155.

\bibitem{CEGN2014}
J.~Ceniceros,
M.~Elhamdadi,
M.~Green and
S.~Nelson,
\emph{Augmented biracks and their homology},
Internat.\ J.\ Math.\ \textbf{25} (2014),
1450087,
19 pp.

\bibitem{FJK2004}
R.~Fenn and
M.~Jordan-Santana and
L.~H.~Kauffman, 
\emph{Biquandles and virtual links},
Topology Appl.\ \textbf{145} (2004),
157--175.

\bibitem{FennRourke92}
R.~Fenn and
C.~Rourke, 
\emph{Racks and links in codimension two},
J.\ Knot Theory Ramifications (4)\textbf{1} (1992),
343--406.

\bibitem{FRS1993}
R.~Fenn,
C.~Rourke and
B.~Sanderson,
\emph{An introduction to species and the rack space}
in Topics in knot theory (Erzurum, 1992),
NATO Adv.\ Sci.\ Inst.\ Ser.\ C Math.\ Phys.\ Sci.\ \textbf{399},
33--55. 

\bibitem{FennRourkeSanderson95}
R.~Fenn,
C.~Rourke and
B.~Sanderson,
\emph{Trunks and classifying spaces},
Appl.\ Categ.\ Structures \textbf{3} (1995),
321--356.

\bibitem{G1984}
D.~Gabai,
\emph{Foliations and genera of links},
Topology \text{23} (1984),
381--394.

\bibitem{GL1989}
C.~Gordon and J.~Luecke
\emph{Knots are determined by their complements},
J.\ Amer.\ Math.\ Soc.\ (2)\textbf{2} (1989),
371--415.

\bibitem{G1953}
V.~K.~A.~M.~Gugenheim,
\emph{Piecewise linear isotopy and embedding of elements and spheres (I) and (II)},
Proc.\ London.\ Math.\ Soc.\ (3)\textbf{3} (1995),
29--53 and 129--152.

\bibitem{HT1985}
A.~Hatcher and
W.~Thurston,
\emph{Incompressible surfaces in 2-bridge knot complements},
Inventiones Mathematicae \textbf{79} (1985),
225--246.

\bibitem{Ishii08}
A.~Ishii,
\emph{Moves and invariants for knotted handlebodies},
Algebr.\ Geom.\ Topol.\ \textbf{8} (2008),
1403--1418.

\bibitem{Ishii15-1}
A.~Ishii,
\emph{A multiple conjugation quandle and handlebody-knots},
Topology Appl.\ \textbf{196} (2015),
492--500.

\bibitem{Ishii15-2}
A.~Ishii,
\emph{The Markov theorem for spatial graphs and handlebody-knots with Y-orientations},
Internat.\ J.\ Math.\ \textbf{26} (2015),
1550116,
23pp.

\bibitem{IIKKMO2018_1}
A.~Ishii,
M.~Iwakiri,
S.~Kamada,
J.~Kim,
S.~Matsuzaki and
K.~Oshiro,
\emph{A multiple conjugation biquandle and handlebody-links}, 
Hiroshima Math.\ J.\ \textbf{48} (2018),
89--117.

\bibitem{IIKKMO2018}
A.~Ishii,
M.~Iwakiri,
S.~Kamada,
J.~Kim,
S.~Matsuzaki and
K.~Oshiro,
\emph{Biquandle (co)homology and handlebody-links}, 
J.\ Knot Theory Ramifications (11)\textbf{27} (2018),
1843011,
1--33.

\bibitem{IIKKMO2021}
A.~Ishii,
M.~Iwakiri,
S.~Kamada,
J.~Kim,
S.~Matsuzaki and
K.~Oshiro,
\emph{Cocycles of $G$-Alexander biquandles and $G$-Alexander multiple conjugation biquandles}, 
Topology Appl.\ \text{301} (2021),
107512,
31 pp.

\bibitem{IshiiIwakiriJangOshiro13}
A.~Ishii,
M.~Iwakiri,
Y.~Jang and
K.~Oshiro,
\emph{A $G$-family of quandles and handlebody-knots},
Illinois J.\ Math.\ \textbf{57} (2013),
817--838.

\bibitem{IMM20}
A.~Ishii,
S.~Matsuzaki and
T.~Murao, 
\emph{A multiple group rack and \oss{s}}, 
J.\ Knot Theory Ramifications (7)\textbf{29} (2020),
2050046,
1--20.

\bibitem{IN2017}
A.~Ishii and
S.~Nelson,
\emph{Partially multiplicative biquandles and handlebody-knots},
Contemp.\ Math.\ \text{689} (2017),
15--176.

\bibitem{Joyce82}
D.~Joyce,
\emph{A classifying invariant of knots, the knot quandle},
J.\ Pure Appl.\ Alg.\ \textbf{23} (1982),
37--65.

\bibitem{K1992}
O.~Kakimizu,
\emph{Finding disjoint incompressible spanning surfaces for a link},
Hiroshima Math.\ J.\ \textbf{22} (1992),
225--236.

\bibitem{K2005}
O.~Kakimizu,
\emph{Classification of the incompressible spanning surfaces for prime knots of 10 or less crossings},
Hiroshima Math.\ J.\ \textbf{35} (2005),
47--92.

\bibitem{Matsuzaki19}
S.~Matsuzaki,
\emph{A diagrammatic presentation and its characterization of non-split compact surfaces in the 3-sphere},
J.\ Knot Theory Ramifications (9)\textbf{30} (2021),
2150071,
1--32.

\bibitem{M1990}
M.~Motto,
\emph{Inequivalent genus $2$ handlebodies in $S^3$ with homeomorphic complement},
Topology Appl.\ \text{36} (1990),
283--290.

\bibitem{KR2003}
L.~H.~Kauffman and
D.~E.~Radford,
\emph{Bi-oriented quantum algebras, and generalized Alexander polynomial for virtual links},
Contemp.\ Math.\ \textbf{318} (2003),
113--140.

\bibitem{Matveev82}
S.~V.~Matveev,
\emph{Distributive groupoids in knot theory},
Mat.\ Sb.\ (N.S.) (161)\textbf{119} (1982),
78--88.

\bibitem{M2005}
T.~Mochizuki,
\emph{The 3-cocycles of the Alexander quandles $\mathbb{F}[T]/(T-\omega)$},
Algebraic Geom.\ Topol.\ \textbf{5} (2005),
183--205.
\end{thebibliography}
\end{document}